\documentclass[a4paper,10pt, oneside]{amsart}
\usepackage[utf8]{inputenc}

\usepackage{geometry}
\usepackage{enumitem}

\usepackage{amsmath, amssymb, amsfonts}
\usepackage{mathtools}
\usepackage{bbm}
\usepackage{csquotes}
\usepackage{mathrsfs}
\usepackage{shuffle}

\usepackage{caption} %
\usepackage{hyperref}
\usepackage{hypcap} %

\usepackage{multirow}
\usepackage{makecell}

\newcommand{\Asig}{\mathcal{A}_{\mathrm{lin}}}
\newcommand{\Alog}{\mathcal{A}_{\mathrm{DNN}}}
\newcommand{\Tsig}{\mathcal{T}_{\mathrm{lin}}}
\newcommand{\Tlog}{\mathcal{T}_{\mathrm{DNN}}}

\newcommand{\R}{\mathbb{R}}
\newcommand{\N}{\mathbb{N}}
\newcommand{\Z}{\mathbb{Z}}

\newcommand{\E}{\mathbb{E}}

\newcommand{\B}{\mathbb{B}}
\newcommand{\X}{\mathbb{X}}
\newcommand{\Y}{\mathbb{Y}}

\newcommand{\bx}{\mathbf{x}}
\newcommand{\by}{\mathbf{y}}

\newcommand{\ba}{\mathbf{a}}
\newcommand{\bb}{\mathbf{b}}
\newcommand{\bg}{\mathbf{g}}

\renewcommand{\P}{\mathbb{P}}
\newcommand{\F}{\mathcal{F}}

\renewcommand{\d}{\mathrm{d}}

\newcommand{\Lip}{\mathrm{Lip}}

\newcommand{\rps}{\Omega^p_T}
\newcommand{\rpsz}{\Omega^{p,0}_T}
\newcommand{\hatrpsz}{\hat{\Omega}^{p,0}_T}
\newcommand{\rpst}{\Omega^p_t}
\newcommand{\rpszt}{\Omega^{p,0}_t}

\DeclareMathOperator{\Sigop}{Sig}
\newcommand{\Sig}[1]{\Sigop(#1)}

\newcommand{\integer}[1]{\lfloor #1 \rfloor}

\makeatletter 
\@mparswitchfalse%
\makeatother
\normalmarginpar

\usepackage{amsthm}

\newcounter{cprop}[section]

\newtheorem{theorem}[cprop]{Theorem}
\newtheorem*{theorem*}{Theorem}

\theoremstyle{plain}

\newtheorem{corollary}[cprop]{Corollary}
\newtheorem*{corollary*}{Corollary}

\newtheorem{lemma}[cprop]{Lemma}
\newtheorem{proposition}[cprop]{Proposition}

\numberwithin{equation}{section}

\theoremstyle{definition}

\theoremstyle{remark}
\newtheorem{remark}[cprop]{Remark}


\title[Stochastic Control with Signatures]{Stochastic Control with Signatures}

\begin{document}

\ifpdf
\hypersetup{
  pdftitle={Stochastic Control with Signatures},
  pdfauthor={P. Bank, C. Bayer, P. P. Hager, S. Riedel, and T. Nauen}
}
\fi

\author{P. Bank}
\address{Peter Bank \\
Institut f\"ur Mathematik, Technische Universit\"at Berlin, Germany}
\email{bank@math.tu-berlin.de}

\author{C. Bayer}
\address{Christian Bayer \\ Weierstrass Institute, Berlin, Germany}
\email{christian.bayer@wias-berlin.de}

\author{P.P. Hager}
\address{Paul P. Hager \\
Department of Statistics and Operations Research, University of Vienna, Austria }
\email{paul.peter.hager@univie.ac.at}

\author{S. Riedel}
\address{Sebastian Riedel \\
Fakult\"at für Mathematik und Informatik, FernUniversit\"at Hagen, Germany}
\email{sebastian.riedel@fernuni-hagen.de}

\author{T. Nauen}
\address{Tobias Nauen \\
Deutsches Forschungszentrum f\"ur K\"unstliche Intelligenz, Kaiserslautern,\break  Germany}
\email{tobias\_christian.nauen@dfki.de}

\begin{abstract}
    This paper proposes to parameterize open loop controls in stochastic optimal control problems via suitable classes of functionals depending on the driver's path signature, a concept adopted from rough path integration theory. We rigorously prove that these controls are dense in the class of progressively measurable controls and use rough path methods to establish suitable conditions for stability of the controlled dynamics and target functional. These results pave the way for Monte Carlo methods to stochastic optimal control for generic target functionals and dynamics. We discuss the rather versatile numerical algorithms for computing approximately optimal controls and verify their accurateness in benchmark problems from Mathematical Finance.
\end{abstract}

\maketitle

\section{Introduction}
Stochastic optimal control approaches typically start by considering the dynamics of a controlled system driven by some source of randomness, originating for instance from a Brownian motion. Controls are then specified in open loop form as a function(al) adapted to this driving randomness or constructed in closed form as feedback functions dynamically applied to the current state of the system. The evolution of the controlled system and the control jointly result in some cost or reward whose expectation is ultimately sought to be optimized.

A key mathematical tool is the dynamic programming principle which in its infinitesimal form as a Hamilton-Jacobi-Bellman (HJB) equation for the value function often takes the form of a nonlinear (integro-)partial differential equation on state space; see \cite{fleming2006stochastic} for a standard reference of this approach. Typically, these equations can be solved only numerically, and even then it remains a daunting task, particularly due to the well known ``curse of dimensionality'':  If one tries to approximate the value function by computing its value in certain points, one quickly needs an impossibly large number of these. Alternative methods where approximations of the value function are parameterized through, e.g., a deep neural network (DNN) have been successfully applied in recent years also in higher dimensions (cf. the survey \cite{germain2023neural} and \cite{jentzen2017deep, sirignano2018dgm, hure2020deep} to name just a few). But even these approaches reach their limits in situations where non-Markovian noise as generated, for instance, by a fractional Brownian motion requires one to store the whole history of the controlled system (or of at least parts of it). Even the derivation of the HJB equation can become challenging as one works on an infinite-dimensional state space; see \cite{fabbri2017stochastic} for a recent monograph on optimal control in such functional analytic settings. In a finite-dimensional Markovian setting, parametrizations of closed loop (feedback) controls have been proposed in \cite{gobet2005sensitivity} and successfully extended using DNNs in \cite{han2016deep} (see also \cite{ruf2020neural, han2021recurrent} and the hybrid method \cite{hure2021deep}).
Again the transition to the non-Markovian setting driven, e.g., by a fractional Brownian motion is far from evident as the entire history of the system (and maybe even more) has to be incorporated in the control. Alternatively, bespoke finite-dimensional approximations to the controlled system can be used. But these need to be carefully tailored to the problem under investigation; see, e.g., \cite{baeuerle2020portfolio} and \cite{bayer_breneis23,bayer2019implied}.

By contrast the approach proposed in the present paper uses what is called path signatures of the driving noise to consider suitable parameterizations of open loop policies. (A  closed loop approach has also recently been proposed in \cite{hoglund2023a} on which we further comment below.)
We will prove that for many problems this signature based class of policies is sufficiently rich to include some which are arbitrarily close to optimal. Studied first by Chen \cite{chen1954iterated}, signatures were identified by Lyons \cite{lyons1998differential} as the building block for his by now well developed theory of rough path integration. At any time the signature consists of all the iterated integrals of the components in the path under consideration against each other until the present moment. This infinite-dimensional tensor can be shown to (essentially) encode the full evolution of the path (\cite{hambly2010uniqueness}, \cite{boedihardjo2016signature}). Therefore, truncated versions of the signature can be expected to be an efficient finite-dimensional encoding of path information that can be used by numerical approaches. Even so, the space of progressively measurable policies still corresponds to the typically vast space of measurable, non-anticipative mappings on the path space of the driving noise. It is thus remarkable that even \emph{linear} functionals of the signature can be shown to be dense in the space of continuous functionals over compact path spaces. This is used by \cite{kalsi2020optimal} for a Stone-Weierstrass argument which is also at the heart of our proof of denseness with respect to convergence in measure of linear or DNN-parameterized families of signature based policies in the class of all admissible ones (cf.~Prop.~\ref{prop:dens_sig_controls}). 

For the workhorse model of stochastic optimal control where ones seeks to minimize expected costs like
\begin{align*}
  \E[L(Y^U,U)]=\E\left[\int_0^T f(t,Y^{U}_t,U_t)\d t + g(Y^{U}_T)\right]
\end{align*}
subject to system dynamics such as
\begin{align}\label{eq:main_sde}
Y^{U}_0=y_0, \quad \d Y^{U}_t = b(Y^{U}_t, U_t)\d t + \sigma(Y^{U}_t)\d X_t,
\end{align}
a stability result (Theorem~\ref{thm:diehl_et_al}) shows that this kind of denseness is sufficient for our purposes (given some integrability assumptions).

Remarkably, this remains true when, instead of the usual choice of $X$ as a Brownian motion, one passes to a stochastic rough driver $\X$ such as a fractional Brownian motion $W^H$. As a consequence, even for such non-Markovian, inherently infinite-dimensional drivers the same signature-based policies can be used for computing approximately optimal ones. Indeed, machine learning tools such as stochastic gradient descent become applicable as soon as one can efficiently produce Monte Carlo samples of $L(Y^U,U)$ for a given signature policy $U$. This optimization procedure is therefore highly versatile and largely model-independent: only its dynamics $Y^U$ need to be generated for an open loop control determined from samples of the extended driver's $(t,\X_t)_{t \geq 0}$ (truncated) signature. The latter can for instance be generated even offline for any driver of interest and then re-used for different dynamics driven by the same~$\X$.

We illustrate the effectiveness of this optimization procedure by working out two case studies on linear-quadratic stochastic optimal control problems. In either case a fractional Brownian motion makes these problems challenging to address by dynamic programming methods. In the first one, an analytic formula for the problem's value is nonetheless available and our numerically computed signature policies turn out to approximate it very well. In the second case study the benchmark is provided by a highly original signature-based numerical method due to \cite{kalsi2020optimal} which even manages to transform the optimization problem into a deterministic one that just draws on the \emph{expected} signature. Here again our more versatile approach performs well. In either case, the signature turns out be needed only up to level 3 or 4, which corresponds to a 5 respectively 8-dimensional encoding (at least when using log-signatures) of the extended fractional Brownian motion's full path $(s,W^H_s)_{s \leq t}$.

\subsection*{Related literature}
Initiated by \cite{diehl2017stochastic}, there is ongoing development in the pathwise control of rough differential equations (\cite{allan2020pathwise, chakraborty2024pathwise}) and the associated \emph{rough} HJB equations. Although these works are only tangentially related to our topic, they require similar stability results for the rough dynamic system with respect to the control. To impose minimal assumptions on the control and the regularity of the rough signal, we present a novel stability result.

The theoretical analysis of continuous-time non-Markovian control problems already constitutes a vast body of literature, a large part of which deals with dynamic stochastic control problems where the driver is Markov but the coefficients are path dependent (e.g.\ delayed systems \cite{kolmanovskii1996control}).
We therefore restrict the discussion to works that have an ambition in implementable numerical methods for problems including a non-Markovian driving signal.

The work \cite{leao2024solving} proposes a discretization method that applies to the optimal drift control of a (path-dependent) differential equation driven by fractional Brownian motion with constant volatility.
Exploiting the underlying Brownian filtration they parameterize open loop controls as a functional of a martingale discretization of Brownian motion, which are then trained based on a dynamic programming principle for the (augmented) discretized structure.
Notably the authors are able to prove convergence rates towards the optimal value.
In general, as we comment in more detail below, such discretiziation methods suffer from a proportional increase in the state space dimension.

The recent paper \cite{hoglund2023a} presents a novel method based on neural rough differential equations, originally introduced in \cite{kidger2020neural, morrill2021neural}, which serve as a continuous-time analogue to recurrent neural networks. 
In this sense, the method can be seen as the continuous-time extension of \cite{han2021recurrent}.
Similar to our work, \cite{hoglund2023a} employs a Monte Carlo-based method by directly parametrizing the control, though in their case, the parameterization is done for feedback (closed loop) controls. Signatures appear in \cite{hoglund2023a} as a special case when using linear activation functions, a perspective the authors use to argue for the density of their approximations. In this case, the signature arises from the system's state rather than its driver.
The stochastic optimal control theory for such path-dependent feedback controls is not as well understood as the open loop case though. The main challenge is that in these feedback formulations the information flow on which controls are based is itself depending on the control to be optimized. This leads to a ``chicken-and-egg problem'' whose mathematical treatment is far from complete; see, e.g., the early account  \cite{bensoussan92} as well as the more recent work \cite{cohen2023optimaladaptivecontrolseparable} and the references therein. In fact, to the best of our knowledge, a formal density argument for parametrized feedback policies in these highly nonlinear control spaces is an open question. This is in contrast to our approach with open loop controls which avoids these intricacies and which allows us to prove in a mathematically rigorous manner that the corresponding optimal control value can be approximated by policies based on signatures of the driving noise. 

The paper \cite{kalsi2020optimal} (and later also \cite{cartea2020optimal}) considers strategies given as linear functionals of the dynamically updated signature of the system's driver. It focuses on optimization problems where the target functional  can be expressed as a linear functional given in terms of a strategy dependent shuffle polynomial acting on the expected signature. The shuffle product used there ensures that the space of linear signature functionals forms an algebra, making the class of such target functionals larger than one might be led to believe from the outset, even though this problems class is far from the generic formulation of a stochastic optimal control problem as considered in the present paper.

In \cite{cuchiero2024signature} signature models are used for stochastic portfolio theory, hence, portfolio optimization, see also \cite{futter2023signature} for a very similar approach. The recent works~\cite{cuchiero2023joint,cuchiero2023signaturebased} analyze asset price models (without optimization) given as SDEs with coefficients defined as linear functionals of signatures of underlying stochastic processes. They calibrate such models (jointly) to VIX and SPX options. %

Signatures can also be used for solving path-dependent backwards stochastic differential equations which can be related to stochastic optimal control problems by Pontryagin principle (see e.g.~\cite{BEFQ24}).

As an alternative to signature methods, numerical approximation of stochastic optimal control problems in non-Markovian settings can be used.
We will specifically comment on two such approaches recently explored for specific examples.

The first such approach is to discretize the problem in time and to then enhance the state variable by the whole history, see, for instance, \cite{becker2019deep} for a specific example.
Hence, if the original state processes evolved in $\R^d$ and we are using a time-grid of length $N$, then the enhanced state process evolves in $\R^{dN}$.
Obviously, this potentially magnifies the curse of dimensionality, unless we can rely on powerful techniques for dealing with high-dimensional problems (such as deep learning as used in \cite{becker2019deep} or tensor-trains as used in \cite{bayer2023pricing}).
In principle, the same issue arises for signature methods (i.e., the dimension of the truncated tensor algebra is much higher than the dimension of the original state space), but signatures are empirically seen to be highly efficient encoding of path properties, implying that the dimension of the state space required to effectively approximate the solution is often much lower.
Crucially, the dimension of the truncated signature does \emph{not} depend on the time-discretization.
This is especially relevant when the underlying process is rough (e.g., fractional Brownian motion with low Hurst index), and therefore requires fine discretization of the stochastic optimal control problem.

A second workable approximation strategy can also be based on a \emph{Markovian approximation} of the non-Markovian state process, see, for instance, \cite{baeuerle2020portfolio} for an application to portfolio optimization in rough volatility models, as well as \cite{bayer_breneis23,bayer2019implied} for American options. 
In essence, the idea is to consider the same control problem, but for a surrogate state process, which is a Markov process.
For instance, when the state process is given in terms of a stochastic Volterra equation, there is a general approach \cite{baeuerle2020portfolio,bayer_breneis23} for constructing multi-factor Markovian approximations essentially of Ornstein-Uhlenbeck type.
This approximation will also highly enlarge the dimension of the state space, but, under suitable conditions, the approximation method can be extremely efficient, see \cite{bayer2023weak} for a concrete example.
In any case, these constructions are highly problem specific, unlike the general, model-free architecture proposed in this paper.
Even if a good approximation to the value function -- now expressed in terms of the multi-factor surrogate process -- is available, actually applying the corresponding strategy in the context of the original problem might be difficult, as it would require ``inverting'' the Markovian approximation in a path-wise manner, which is generally not a well-posed problem. 

\subsubsection*{Outline of the paper}
In Section~\ref{sec:problem_description} we formalize the standard control problem studied in the paper, spelling out and discussing in particular the technical assumptions on the coefficients in the dynamics of $Y^U$ and in the cost functional $L=L(Y^U,U)$.

In Section~\ref{sec:notat-basic-defin} we provide an introduction to signatures and provide tools needed from rough path theory. We also describe the natural (deterministic!) rough path setting for the analysis of stochastic optimal control problems, namely \emph{stopped rough paths}; see also \cite{kalsi2020optimal}. We also recall necessary prerequisites from the theory of rough differential equations (RDEs), and 

Section~\ref{sec:continuous_approximation} introduces two signature-based classes of controls to approximate general progressively measurable controls $U$. Specifically, we consider \emph{linear signature policies} in $\Asig$ which are just linear functionals of the present signature. 
Alternatively, \emph{deep signature controls} from $\Alog$ are obtained by using the path's signature (or, more precisely, its log-signature) as input in a neural network. Proposition~\ref{prop:dens_sig_controls} shows that both linear signature controls and deep signature controls are dense in the set $\mathcal{A}$ of admissible controls.
By a new stability result for RDEs with a function-valued parameter (Theorem~\ref{thm:diehl_et_al}) --- think of the control as a function in time --- convergence of the controls implies convergence of the controlled process $\mathbb{Y}^U$ (Lemma~\ref{lem:continuity-Y-controlled}), and this further implies that the associated costs also converge. Theorem~\ref{thm:approximation} then concludes that the infimum of expected costs is the same over the three policy classes $\mathcal{A}$, $\Asig$, and $\Alog$.

Finally, Section~\ref{sec:numerics} gives details of the numerical method suggested, together with a discussion of numerical properties and possible extensions. We then provide two case studies. In the first, the driving process $X$ is a fractional Brownian motion, which we try to force to stay close to $0$ by controlling its drift. In order to obtain a well-posed problem, we penalize by the $L^2$ norm of the control. This problem's value is available analytically in closed form from \cite{bank2017hedging} and can thus be used as a benchmark. We observe excellent accuracy for both linear and deep signature methods.

In the second case study, we revisit an optimal execution problem already studied in \cite{kalsi2020optimal}. Here, the driving process, again chosen to be fractional Brownian motion, corresponds to the fundamental price of a financial asset. We want to liquidate a position in that asset while maximizing the proceeds from selling which are adversely affected by price impact. The bespoke optimization procedure of \cite{kalsi2020optimal} provides another valuable benchmark that we are able to match with our methodology.

\vspace{1em}
\noindent\textbf{Acknowledgments.}
We would like to thank Peter K.~Friz and Imanol Perez-Arribas for helpful discussions. 
Peter Bank and Christian Bayer gratefully acknowledge funding by Deutsche Forschungsgemeinschaft through SFB TRR 388 Project B03.

\section{An optimal control problem of a rough differential system}\label{sec:problem_description}
Let us start by describing more precisely the class of optimization problems we consider in this paper.
We fix an underlying complete filtered probability space  $\left(\Omega, \mathcal{F}, (\mathcal{F}_t)_{t \geq 0}, \P\right)$. \emph{Admissible controls} will be $(\mathcal{F}_t)$-progressively measurable processes $U: [0,T] \times \Omega \to \mathcal{U}$ taking values in %
$\mathcal{U}$, a closed, convex and nonempty subset of $\mathbb{R}^k$.
The set of such admissible controls will be denoted by $\mathcal{A}$. Any realization $u=U(\omega):[0,T] \to \mathcal{U}$ can then be viewed as an element of $L^0(\d{t};\mathcal{U})$, the space of measurable functions $u: [0,T] \to \mathcal{U}$ equipped with the topology of convergence in (Lebesgue) measure $\mathrm{Leb}(\d t)=\d t$.

We study the optimal control of the differential dynamics
\begin{align}\label{eq:main_srde}
\d Y^{U}_t = b(Y^{U}_t, U_t)\d t + \sigma(Y^{U}_t)\d \X_t, \qquad t \in [0,T], \quad Y^{U}_0 = y_0 \in \R^{m},
\end{align}
where $b:\R^{m}\times\mathcal{U} \to \R^{m}$, $\sigma:[0,T]\times \R^{m} \to \R^{m\times d}$ satisfy the Lipschitz assumptions
\eqref{ass:first_on_b} and \eqref{ass:on_b} specified below. The stochastic driver $\X$ is a stochastic rough path with finite $p$-variation for some $p\in[1, \infty)$. The reader can think of $\X$ as a Brownian motion, but also a fractional Brownian motion can be considered. We refer to Section~\ref{sec:preliminaries} for the precise details and an introduction to rough paths including the definition of the metric rough path space $\rpsz(\R^m)$ we will consider for our driver $\X$. 
By Theorem~\ref{thm:diehl_et_al} there exists then a unique solution $Y^{U}$ to $\eqref{eq:main_srde}$ which canonically extends to a stochastic rough path $\mathbb{Y}^U: \Omega \to \rps(\R^m)$ and depends continuously on $y_0$, $\X$ and $U$.
The optimization objective is to minimize the expected costs
$$J(U) := \E[L(\Y^{U}, U)]$$
among all admissible controls $U\in \mathcal{A}$, where the \emph{cost functional}
\begin{flalign}\label{ass:cost_functional} \text{ $L:\R^m \times \Omega_T^{p} \times L^0(\d{t};\mathcal{U}) \to \R$ is continuous and bounded.}&&
\end{flalign}
It is then clear that there is a finite \emph{optimal cost}
\begin{align*}
\inf_{U\in\mathcal{A}}J(U).
\end{align*}
In Section~\ref{sec:signature_strategies} below we will characterize subclasses of admissible controls, sufficient for approximating the optimal cost. In fact, assuming additionally that
\begin{flalign}\label{ass:filtration}
\text{ $\F_t$ is the completion of $\sigma(\X_s \;\vert\; 0 \le s \le t)$ by all $\P$-null sets of $\F$, $0 \leq t \leq T,$ }&&
\end{flalign}
it will turn out that it suffices to restrict to controls that are continuous functions on the path-space $\Omega^p_T$. This fact allows us to deduce that it also suffices to restrict to controls that are (linear) functionals of the signature of the time-augmented path $\widehat\X$ defined below.

In the rest of this section we will discuss the above assumptions in more detail and comment on some possible generalizations.

\subsection*{Form of the dynamics}
Drift-controlled stochastic differential equations of the form \eqref{eq:main_sde} find many applications in various fields.
Studying the rough path extension \eqref{eq:main_srde} allows us to treat non-semimartingale drivers $\X$ in a unified framework, including non-Markovian processes such as fractional Brownian motion.
The price to pay for this generality is that we cannot allow for a controlled diffusion part as in
\begin{align}\label{eq:vol_controlled}
\d Y^{U}_t = b(Y^{U}_t, U_t)\d t + \sigma(Y^{U}_t, U_t)\d \X_t.
\end{align}
This is because the stochastic part of these dynamics will only be defined rigorously for a general rough path $\X$ if the control $U$ is \emph{controlled} by $\X$, that is, roughly speaking, it cannot fluctuate more than $\X$ itself; see \cite{friz2020course} for an introduction to rough differential equations using \emph{Gubinelli derivatives}. 
A stability result for systems of the form \eqref{eq:vol_controlled} similar to Theorem~\ref{thm:diehl_et_al} has recently been established in \cite{allan2020pathwise} for the case $p\in[2,3)$ under the assumption that the controls $U$ have finite $p/2$-variation.
In contrast to the present work, the authors in \cite{diehl2017stochastic, allan2020pathwise} have considered a pathwise ``anticipating'' control of the system \eqref{eq:main_srde}.
In \cite{diehl2017stochastic} the restriction to drift-controlled system was imposed for non-degeneracy reasons, while \cite{allan2020pathwise} included control in the volatility
by introducing a penalty term that ensures the required regularity of controls similar to a Tychonov regularization.

We stress that the results from Section~\ref{sec:continuous_approximation} are not tied to the stability results for rough differential equations.
Instead, we can study the optimal control of stochastic differential equations for which we obtain stability under weaker conditions. In fact, referring to \cite{protter2005stochastic} one even has sufficient stability for path-dependent stochastic differential equations of the form
\begin{align}\label{eq:functional_sde}
    \d{Y_t} = b(U_t, Y\vert_{[0,t]})\d{t} + \sigma(U_t, Y\vert_{[0,t]})\d{X_t},
\end{align}
when $X$ is a continuous semimartingale, assuming that $b$ and $\sigma$ are bounded and \emph{functional} Lipschitz uniformly in $U$.
Thus, to adapt our approximation results from Section~\ref{sec:continuous_approximation} to cover the optimal control of \eqref{eq:functional_sde} we only need to modify the assumption on the cost functional.
More precisely, we need to require the continuous dependence of $L$ on $Y$, not in the rough path sense, but in the uniform topology on $C([0,T];\R^m)$ (see also the comments on the form of the cost functional below).
This even extends to the case of càdlàg semimartingales, where the canonical lift is given by the iterated Marcus-integral (see \cite{friz2027general} and \cite{chevyrev2019canonical}) and corresponding universal approximation theorems have recently been provided in \cite{cuchiero2024universal}.

\subsection*{Form of the cost functional} 
A general form of cost functionals satisfying the continuity assumption in \eqref{ass:cost_functional} is given by
\begin{align}\label{eq:running_costs}
L(\by, u):= \int_0^{T} f(t, y_t, u_t) \d{t} + g(y_T),  \quad \by\in\Omega_T^p, \quad u \in L^0(\d{t};\mathcal{U}),
\end{align}
where $f:[0,T]\times\R^{m}\times \mathcal{U} \to \R$, $g:\R^{m} \to \R$  are continuous and bounded.
In view of such ``classical'' cost functionals the definition of $L$ as a map on the rough path space may seem cumbersome. 
However, this definition does not lead to additional complications and allows us to consider cost functionals involving rough integrals such as
\begin{align*}
L(\by, u):= g\left(\int_0^{T} f(y_t) \d{\by_t}\right), \qquad \by\in\Omega_T^p, \quad u \in L^0(\d{t};\mathcal{U}),
\end{align*}
where $f\in\mathrm{Lip}^\gamma(\R^m\times\mathcal{U}; \R)$ for some $\gamma > p$ (see the next section for the definition of this space) and $g:\R\to\R$ is continuous and bounded (see \cite[Section  10.6]{FV10} for definition of the rough integral and its continuity properties).

\subsection*{Boundedness of the cost functional}
Assumption \eqref{ass:cost_functional} allows us to disentangle the problem of approximating a given control from the convergence of the associated costs.
More precisely, we prove the approximation of admissible controls in $\mathcal{A}$ by controls in a subclass $\mathcal{A}^{\prime} \subset \mathcal{A}$ for almost sure convergence in $L^0(\d{t}; \mathcal{U})$.
The convergence of the associated costs then follows from the stability result in Theorem~\ref{thm:diehl_et_al} and the dominated convergence theorem.
If one were to consider unbounded cost functionals this disentanglement would in general not be possible.
Firstly, the set of admissible controls is intrinsically related to the cost functional by $\mathcal{A} \subset \{U \text{ prog. meas.} \;\vert\; \E[\vert L(\Y^U, \X)\vert] < \infty\}$.
Secondly, the mode of convergence for approximating controls needs to be chosen accordingly to deduce the convergence of costs. Hence, the theoretical analysis becomes more problem-specific and possibly requires more elaborate universal approximation results for signature functionals.

To be more specific about the last statement, consider a quadratic control problem with a cost functional of the form
\begin{align*}
L(\by, u):= \int_0^{T} \left\{ a\vert y_t \vert^2 + b\vert u_t\vert^2 \right\} \d{t},  \quad \by\in\Omega_T^p, \quad u \in L^0(\d{t};\mathcal{U}),
\end{align*}
where $a, b >0$ are constants and $\mathcal{U} = \R^k$.
We consider examples of such problems in Section~\ref{sec:optimal_tracking} and \ref{sec:optimal_execution}.
In this case, assuming square-integrability of $\Vert \X \Vert_{p-\mathrm{var}}$, one needs to prove convergence of the approximating sequence of controls in ${L^2(\d{t}\otimes\d\P)}$.
For the sub-class of signature controls this cannot be easily deduced from the universal approximation in Proposition~\ref{prop:universial_signature_stopped}, since outside the chosen compact sets there is no general way to bound the costs incurred.
Note, however, that recently global universal approximation theorems have become available for robust (normalized) signatures in \cite{chevyrev2022signature} and on weighted functions spaces in  \cite{cuchiero2023global}.
In particular, \cite{bayer2023primal}  provide a corresponding universal approximation for processes in ${L^2(\d{t}\otimes\d\P)}$, thus allowing to directly adapt our theoretical results to the case of quadratic control problems.

\subsection*{Initial value and history} 
The assumption that $\X$ takes values in $\rpsz$ implies a deterministic starting value $\X_0$.
A straight forward generalization is an inclusion of a finite history of the path, i.e., by assuming that $\X$ is defined on the interval $[t_0, T]$ for some $t_0 < 0$ and then setting $\F_t$ to the completion of $\sigma(\X_s \;\vert\;t_0 \le s \le t)$ for all $t\in[0,T]$.
In this case, the approximation results from Section~\ref{sec:continuous_approximation} are extended by considering controls that are functionals of the signature started at time $t_0$.

\section{Preliminaries on rough analysis and signatures}\label{sec:preliminaries} 
\label{sec:notat-basic-defin}

We are going to introduce the basic definitions and notation needed for the understanding rough differential equations and signatures. These definitions are standard in the rough path literature, we refer to \cite{LCL07, friz2020course,FV10} for a more detailed exposition.

\subsection{The tensor algebra}\label{sec:tensor}
The sequence of iterated integrals of a smooth path satisfies algebraic relations that are consequence of the linearity of the integral and the integration by parts rule.
These algebraic properties are most conveniently revealed when organizing these intergals as  \emph{tensor series}.
In this section we introduce the basic algebraic concepts that will allow us to define signatures and rough paths.

 Let $V$ be a finite-dimensional $\R$-vector space with basis $\{e_1,\ldots,e_d\}$. We define the \emph{extended tensor algebra} by setting
 \begin{align*}
   T((V)) \coloneqq \prod_{n = 0}^{\infty} V^{\otimes n}
 \end{align*}
 where $V^{\otimes n}$ denotes the $n$-th tensor power of $V$ with the convention $V^{\otimes 0} \coloneqq \R$, ${V^{\otimes 1} \coloneqq V}$. 
 The algebraic structure on $T((V))$ is given by componentwise summation and tensor multiplication, i.e., for two tensor series $\ba = (a_n)_{n = 0}^{\infty}$ and $ \bb = (b_n)_{n = 0}^{\infty}$ in $T((V))$ the tensor product is defined by
 \begin{align*}
     \ba \otimes \bb = \left(\sum_{k=0}^n a_k \otimes b_{n-k} \right)_{n=0,1,\dots}.
 \end{align*}
 We denote by $\mathbf{0} \coloneqq (0, 0, \dots )$ and $\mathbf{1} \coloneqq (1, 0, \dots )$ the neutral elements of summation and multiplication.
 The \emph{tensor algebra} $T(V) \subset T((V))$ consists of those tensor series with only finitely many non-zero elements, which defines a subalgebra.
 There is a natural pairing between $T((V))$ and $T(V)$ given by
 \begin{align*}
    \langle \cdot,\cdot \rangle \colon T(V) \times T((V)) \to \R, \quad \langle \ba, \bb\rangle := a_0b_0 + \sum_{k=1}^\infty \langle a_k, b_k\rangle,
 \end{align*}
 where the summation is finite by the definition of $T(V)$ and $\langle a_k, b_k \rangle$ denotes the dot product\footnote{This inner product on $V^{\otimes k}$ is defined as the bilinear extension of $\langle e_{i_1} \otimes \cdots \otimes e_{i_k}, e_{j_1} \otimes \cdots \otimes e_{j_k}\rangle := \langle e_{i_1}, e_{j_1}\rangle \cdots \langle e_{i_k}, e_{j_k}\rangle := \delta_{i_1, j_1} \cdots \delta_{i_k, j_k}$, where $\delta$ denotes the Kronecker-Delta.} on $V^{\otimes k}$.
 We will frequently use this pairing by associating an element $\ell \in T(V)$ to the linear functional on the extended tensor algebra $\langle \ell, \cdot\rangle : T((V)) \to \R$.

 The \emph{truncated tensor algebra} is defined by
 \begin{align*}
  T^{N}(V) \coloneqq \bigoplus_{n = 0}^{N} V^{\otimes n}.
 \end{align*}
 We define maps $\pi_n \colon T((V)) \to V^{\otimes n}$ and $\pi_{\leq N} \colon T((V)) \to T^{N}(V)$ by $\pi_n(\mathbf{a}) = a_n$ and $\pi_{\leq N}(\mathbf{a}) = (a_0,\ldots,a_N)$.
Note that $T^{N}(V)$ forms an algebra under the truncated tensor multiplication $\mathbf{a}\,{\otimes}_{N}\,\mathbf{b}$ $:=$ $\pi_{\le N}(\mathbf{a}\otimes\mathbf{b}),$  for $\mathbf{a},\mathbf{b}\in T^N(V)$. However, we will not distinguish between the multiplication symbols on $T^N(V)$, $T(V)$ and $T((V))$ and use $\otimes$ in all cases and write $\mathbf{0}$ and $\mathbf{1}$ for the neutral elements $\pi_{\leq N}(\mathbf{0})$ and $\pi_{\leq N}(\mathbf{1})$ in the truncated tensor algebra.
Furthermore, $T^N(V)$ is a finite dimensional vector space which we equip with the norm
\begin{align*}
    \vert \ba \vert = \sum_{n=0}^N |a_n|, \qquad \ba \in T^N(V),
\end{align*}
where $\vert \cdot \vert = \sqrt{\langle \cdot, \cdot \rangle}$ denotes the Euclidean norm on $V^{\otimes n} \cong \R^{d^n}$ for any $n\in\N$.

 The Lie-algebra generated from $\{\mathbf{e}_1, \dots, \mathbf{e}_d\}$, where $\mathbf{e}_i :=(0, e_i, 0, \dots) \in T(V)$, with the commutator bracket $[\ba, \bb] = \ba\otimes\bb - \bb\otimes\ba$ for $\ba, \bb \in T(V)$ is called the \emph{free Lie-algebra} $\mathfrak{g}(V)$ over $V$. Its extension $$\mathfrak{g}((V)) = \{\bx \in T((V))\;\vert\; \pi_{\le N}(\bx) \in \mathfrak{g}(V) \text{ for all }N\ge 0\}$$ is called the set of free Lie-series over $\{\mathbf{e}_1, \dots, \mathbf{e}_d\}$ and forms a sub algebra of the set of tensor series starting with zero $T((V))_0 = \{\ba \in T((V)) \;\vert\; \pi_0(\ba) = 0\}.$
 The tensor exponential of $\mathfrak{g}((V))$,
 i.e., its image under the map
 \begin{align}\label{eq:tensor_exp}
     \exp_\otimes: T((V))_0 \to T((V)), \quad \mathbf{a} \mapsto \mathbf{1} +\sum_{n=1}^\infty \frac{1}{n!}\ba^{\otimes n},
 \end{align}
is called the group of \textit{group-like elements} $G(V) := \exp_\otimes(\mathfrak{g}((V)))$, which is a subgroup of $T((V))_1 = \{\ba \in T((V)) \;\vert\; \pi_0(\ba) = 1\}$.
Indeed, $(G(V), \otimes)$ is a group with identity $\mathbf{1}$, where for $\bg = \exp_{\otimes}(\ba)\in G(V)$ the inverse element is given by $\bg^{-1} = \exp_\otimes(-\ba)$.
The inverse of the operation \eqref{eq:tensor_exp} is given by the tensor logarithm
  \begin{align}
     \log_\otimes: T((V))_1 \to T((V))_0, \quad (\mathbf{1}+\mathbf{a}) \mapsto \mathbf{1} +\sum_{n=1}^\infty \frac{(-1)^{n+1}}{n}\ba^{\otimes n}.
 \end{align}
 
 We also set $\mathfrak{g}^N(V) := \pi_{\leq N}(\mathfrak{g}(V))$ and $G^N(V) \coloneqq \pi_{\leq N}(G(V))$, which are the \emph{free nilpotent Lie algebra} and \emph{group} of order $N$. We equip these spaces with the subspace topology in $T^N(V)$.
 The truncated tensor exponential and logarithm $$\exp^N_\otimes: T^N(V)_0 \to T^N(V)_1, \qquad \log^N_\otimes: T^N(V)_1 \to T^N(V)_0,$$
 are defined using the corresponding (finite!) power series in the truncated tensor algebra.
 Hence, these maps are smooth, and, furthermore, it holds $\log_\otimes^N = (\exp_\otimes^N)^{-1}$ and  $G^N(V) = \exp^N_\otimes(\mathfrak{g}^N(V))$.

 \subsection{Rough paths and their signatures}\label{sec:rough_paths}

 One can start by defining the signature for a piecewise smooth path $x:[0,T] \to V$, i.e., assume that the derivative $\dot{x}$ is integrable and has only finitely many points of discontinuity.
 We define its signature $\Sig{x}_{s,t}$ over the interval $[s,t] \subset [0,T]$ as the tensor series of iterated integrals. More precisely, $\Sig{x}_{s,t} \in T((V))_1$ and for all $n=1,2,\dots$ we set
\begin{align*}
\pi_n(\Sig{x}_{s,t})&\coloneqq\int_{s}^t \int_s^{t_k} \cdots \int_{s}^{t_2} \dot{x}(t_1) \otimes\cdots \otimes \dot{x}(t_n) \, \d{t_1} \cdots \d{t_n}.
\end{align*}
One can show that as consequence of the integration by parts rule such signatures are elements of the free Lie group $G(V)$, i.e., $\Sig{x}: \Delta_T \to G(V)$, where $\Delta_T \coloneqq \{(s,t) \in [0,T]^2 \, :\, 0 \leq s \leq t \leq T\}$. Moreover, we have the Chen relation (see \cite{chen1957integration})
\begin{align}\label{eq:chen}
    \Sig{x}_{s,t} = \Sig{x}_{s,u} \otimes \Sig{x}_{u,t}, \qquad 0 \le s \le u \le t \le T.
\end{align}
We set $\Sig{x}_t \coloneqq \Sig{x}_{0,t}$ and also define the truncated signature by $\Sig{x}^{\le N}\coloneqq \pi_{\le N}(\Sig{x})$.

The space $\rps(V)$ of \emph{geometric $p$-rough paths} for some $p\ge1$ is constructed as the closure of the set of smooth paths under a $p$-variation metric that also measures the first $\integer{p}$-levels of the signature, where $\integer{p}$ denotes the integer part of $p$.
To make this more precise, we first define the \emph{$p$-variation}-distance of two group valued paths $\bx, \by \colon [0,T] \to G^{N}(V)$ by
\begin{align*}
 d_{p-\mathrm{var};[s,t]}(\bx,& \by) \coloneqq \max_{k = 1,\ldots,N} \sup_{\mathcal{D} \subset [s,t]} \left( \sum_{t_i \in \mathcal{D}} \Big\vert\pi_k\Big(\bx(t_i)^{-1}\otimes\bx(t_{i+1}) - \by(t_i)^{-1}\otimes\by(t_{i+1})\Big)\Big\vert^{\frac{p}{k}} \right)^{\frac{k}{p}}
\end{align*}
where the supremum ranges over all partitions $\mathcal{D}=\{t_i\}$ of $[s,t]$.
We use the notation $d_{p-\mathrm{var}}(\bx, \by) \coloneqq d_{p-\mathrm{var};[0,T]}(\bx, \by)$.
The space of geometric $p$-rough paths $\rps(V)$ is then defined as the set of continuous paths $\bx \colon [0,T] \to G^{\integer{p}}(V)$ that satisfy
\begin{enumerate}[label=(\roman*)]
    \item $d_{p- \mathrm{var};[s,t]}(\mathbf{1}, \bx)<\infty$,%
    \item there exists a sequence of piecewise smooth paths $(x_n)_{n\ge1}$ with $x_n : [0,T] \to V$ such that $$\lim_{n\to\infty} d_{p-\mathrm{var}}\big(\bx,\, \Sig{x_n}^{\le \integer{p}}\big) = 0.$$
\end{enumerate}
Unless stated otherwise we will abbreviate $\rps = \rps(V)$.
On the subset
$$\rpsz := \{ \bx \in \rps \;\vert\; \bx(0) = \mathbf{1}\},$$
$d_{p-\mathrm{var}}$ defines a metric. We can similarly equip $\rps$ with a metric by additionally comparing starting values.
With the thus obtained topology $\rpsz$ and $\rps$ are Polish spaces.

Lyons Extension Theorem (see \cite[Theorem 3.7]{LCL07}; for closer terminology, see \cite[Theorem 9.5]{FV10}) states that to every geometric rough path $\bx \in \rpsz$ we can associate a unique map $\Sig{\bx}: \Delta_T \to G(V)$ (also called the full lift) that satisfies the Chen relation \eqref{eq:chen} and is such that $\Sig{\bx}^{{\leq \integer{p}}}_{s,t} \coloneqq \pi_{\leq \integer{p}}(\Sig{\bx}_{s,t}) = \bx(s)^{-1}\bx(t)$ for all $(s,t)\in\Delta_T$.
Furthermore, \cite[Corollary~9.11]{FV10} implies that for every $N \ge \integer{p}$
the truncated signature map $$\Omega^{p, 0}_T \to G^{N}(V), \qquad \bx\mapsto \Sig{\bx}^{\le N}_{0,T}$$ is continuous. 
Note that for a piecewise smooth path $x:[0,T] \to V$ we have $\Sig{x}_{0,\cdot}^{\le \integer{p}} \in \rpsz$ for any $p \ge 1$. Furthermore, the signature in the rough path sense is consistent with signature in the smooth sense, i.e., $\Sig{\Sig{x}^{\le \integer{p}}} = \Sig{x}$.

We say that a continuous path $x: [0,T] \to V$ has a \emph{lift} to a geometric $p$-rough path, if there exists $\bx \in \rps$ such that $\pi_1(\bx) = x$.
Note that this lift is not unique, even though there often is a canonical choice. For instance, this is the case for solutions to rough differential equations (see the Section~\ref{sec:rde_intro}) and for many stochastic process.
The canonical lift of a Brownian motion $B$, for example, is obtained by suitable piecewise linear approximations $(B^n)_{n\ge 1}$. It can be shown that for any $p \in (2,3)$ we have %
\begin{align*}
    \Sig{B^n}^{\le 2} ~\xrightarrow[n \to \infty]{\rps}~ \B := \left(1, B, \left(\int_{0}^{\cdot} B^i_s  \circ \d{B^j_s}\right)_{i,j=1, \dots, d}\right) \qquad \text{ almost surely}, 
\end{align*}
where ``$\circ \d{B}$'' denotes the Stratonovich integration with respect to $B$.
In particular, $\B$ defines a $\rps$ valued random-variable, which we call the \emph{Stratonovich-lift} of the Brownian motion $B$.

\subsection{Stopped rough paths}\label{sec:stopped_rp}
We can define a stochastic rough path as an $\rps$-valued random variable.
To introduce the notion of adaptedness,  we will need to consider functionals of the restriction of a rough path to subintervals of $[0,T]$.
This calls for the definition of an appropriate ambient space $\Lambda_T$, where we can compare paths defined on different segments of the time line.
For rough paths this was originally proposed in \cite{kalsi2020optimal}, then elaborated in \cite{BHRS23}, and is motivated by the functional Itô calculus, see \cite{Dup19, ContFournie10}.

We call $\Lambda _T \coloneqq \bigcup_{t \in [0,T]} \rpszt$ the \emph{space of stopped rough paths}. Note that
\begin{align*}
\Lambda _T  %
&= \{ \bx\vert_{[0,t]} \;\vert\; \bx \in \rpsz, \quad t\in [0,T]\}.
\end{align*}
Following \cite{ContFournie10} and \cite{cuchiero2024signature} we also call a map $f: \Lambda_T^p \to \R$ a \emph{non-anticipative functional}. To indicate from which set  $\rpszt$ and element $\bx \in \Lambda_T$ is chosen, we will write it as $\bx=\bx^t$.

To define a suitable topology on $\Lambda_T$, note that we can extend a path segment $\bx^t \in \rpszt$ by constant extrapolation $\bx^t(\cdot \wedge t) \in \rpsz$.
We equip $\Lambda_T$ with the metric
  \begin{align*}
    d(\bx^t, \by^s) \coloneqq d_{p-\mathrm{var}}\big(\bx^t(\cdot \wedge t), \by^s(\cdot \wedge s)\big) + |t-s|, \qquad \bx^t, \by^s \in \Lambda_T.
  \end{align*}

Clearly the maps $\Lambda_T \to \rpsz$, $\bx^{t}\mapsto \bx^t(\cdot \wedge t)$ and $\Lambda_T \to [0,T]$, $\bx^{t}\mapsto t$ are continuous and it can be seen that $\Lambda_T$ inherits some of the topological properties of $\Omega^{0,p}_T$.
 More precisely, we have the following result from \cite[Appendix~A]{BHRS23}.
 
\begin{proposition}\label{prop:lambda_topology}
    $\Lambda_T$ is a Polish space. Furthermore, the map $\varphi: [0,T]\times\rpsz \to \Lambda_T$, $(t,\bx) \mapsto {\bx}\vert_{[0,t]}$ is continuos.
\end{proposition}

\subsection{Time augmention and universal approximation}\label{sec:augmented_rp}

In order to use the signature as a feature  on the path space, we extend paths by a running time component.
This will guarantee that the signature uniquely characterizes the path on the entire interval.
In this section we will explain how this extension works for geometric rough paths. We also explain the universality of linear signature maps; see Proposition \ref{prop:universial_signature_stopped}. 

Let $\bx \in \rpsz$ and let $(x_n)_{n\ge1}$ be a sequence of piecewise smooth paths such that $$\lim_{n\to\infty} d_{p-\mathrm{var}}\big(\bx,\, \Sig{x_n}^{\le \integer{p}}\big) = 0.$$
We extend $x_n$ to a piecewise smooth path $\hat{x}_n := t\mapsto (t, x_n(t))$.
It then follows from \cite[Theorem 9.30]{FV10} that $(\Sig{\hat{x}_n}^{\le\integer{p}})_{n\ge1}$ forms a Cauchy sequence in $\rpsz(\R\times V)$ and we denote by $\hat{\bx}$ its limit.
Furthermore, we see from the same theorem that there exists a constant $C>0$ depending on $p$ and $T$ such that 
\begin{align*}
    d_{p-\mathrm{var}}(\bx, \by) \le d_{p-\mathrm{var}}(\hat\bx, \hat\by) \le C d_{p-\mathrm{var}}(\bx, \by), \qquad \bx, \by \in \rpsz.
\end{align*}
We define by $\hatrpsz := \{ \hat{\bx} \;\vert\; \bx \in \rpsz\} \subset \rpsz(\R\times V)$ the set of \emph{time augmented geometric $p$-rough paths}.
The above estimate implies that the embedding $\bx \mapsto \hat{\bx}$ is continuous and that $\hatrpsz$ is also a Polish space. We extend the notions from the previous section to these time augmented paths by passing from the stopped rough path $\bx^t \in \Lambda_T$ to its augmented version $\hat{\bx}^t \in \hatrpsz$.

We will often use that signatures of time-augmentations depend continuously on the original rough path:
\begin{lemma}\label{lem:continuity_timesig}
    For any $N\in \mathbb{N}$, $\Lambda_T\to G^{N}(V): \; \bx^t \mapsto \Sig{\hat{\bx}^t}^{\le N}_{0,t}$ is continuous.
\end{lemma}
\begin{proof} For $\bx^t\in\Lambda_T$ define the extrapolated path $\bx^{t,T} := \bx^t(t \wedge \cdot) \in \rpsz$ and its time-augmentation by $\hat{\bx}^{t,T} \in \hatrpsz$.
We then have by Chen's identity
\begin{align*}
\Sig{\hat{\bx}^t}_{0,t} &= \Sig{\hat{\bx}^{t,T}}_{0,t} \\
&= 
\Sig{\hat{\bx}^{t,T}}_{0,T} \otimes
(\Sig{\hat{\bx}^{t,T}}_{t,T})^{-1} \\
&= 
\Sig{\hat{\bx}^{t,T}}_{0,T} \otimes
\exp_{\otimes}(-e_1(T-t))
\end{align*}
The statement then follows from continuity of the maps $\bx^t \mapsto (t, \bx^{t,T})$, $\bx \mapsto \hat{\bx}$ and $\bx \mapsto \Sig{\bx}^{\le N}_{0,T}$ for any $N\in\mathbb{N}$ are continuous.
\end{proof}

We are now ready to state the approximation property of linear signature functionals within the class of continuous non-anticipative functionals.

\begin{proposition}\label{prop:universial_signature_stopped}
    For any continuous function $f: \Lambda^p_T \to \R$ and compact set $K\subset \rpsz$ there exists a sequence $(\ell_n)_{n\ge1}\subset T(V)$ such that
    \begin{align*}
        \lim_{n\to\infty} \sup_{t\in[0,T]} \sup_{\bx \in K}\Big( \big\vert f(\bx\vert_{[0,t]}) - \langle \ell_n, \Sig{\hat{\bx}\vert_{[0,t]}} \rangle \big\vert \Big) = 0.
    \end{align*}
\end{proposition}
This form of universal approximation property goes back to \cite{kalsi2020optimal}. Let us sketch its proof since it highlights key structural properties of signatures which underline their usefulness for encoding path information.
From the algebraic properties of the signature, it follows that the linear functionals $\{\bx^t \mapsto \langle \ell, \Sig{\hat{\bx}^t}_{0,t}\rangle\;\vert\; \ell \in T(V)\}$ form an algebra (cf. \cite[Theorem 2.15]{LCL07}).
By Lemma~\ref{lem:continuity_timesig} it is a sub-algebra of $C(\Lambda_T; \R)$.
Since the signature characterizes the path (see \cite{hambly2010uniqueness}, \cite{boedihardjo2016signature}), the algebra is point separating (see \cite[Lemma B.3]{kalsi2020optimal} for detailed proof). 
Finally, by Proposition~\ref{prop:lambda_topology}, we have that $\varphi([0,T] \times K) \subset\Lambda_T$ is compact.
The statement then follows from the Stone-Weierstrass Theorem.

\subsection{Rough differential equations}\label{sec:rde_intro}

Starting as in Section~\ref{sec:rough_paths}, we can give meaning to a rough differential equation
\begin{align}\label{eq:introductary_rde}
    \d{y}(t) = \sigma(y(t)) \d{\bx}(t),\quad t \in [0,T], \qquad y(0) = y_0\in\R^m.
\end{align}
where $\bx \in \rpsz(\R^d)$ and $\sigma: \R^m \to \R^{m\times d}$, by first considering it for piecewise smooth paths.
Indeed, let $x:[0,T] \to \R^d$ be piecewise smooth and consider the system of ordinary integral equations
\begin{align}\label{eq:ode}
    y^k(t) = y_0^k + \sum_{i=1}^d\int_0^t  \sigma_i^k(y(s)) \dot{x}^{i}(s) \d{s}, \qquad k = 1, \dots, m, \quad t \in[0,T], \quad y_0 \in \R^m.
\end{align} 
Given that $\sigma$ is bounded and Lipschitz continuous, the above equation has a unique solution,
 which we denote by $\Gamma_\sigma(y_0; x)$.
We then say that a continuous path $y: [0,T] \to \R^m$ is a solution to the rough differential equation \eqref{eq:introductary_rde} if there exists a sequence of piecewise smooth paths $(x_n)_{n\ge1}$ such that
\begin{enumerate}[label=(\roman*)]
    \item\label{item:rde_1} $\lim_{n\to\infty} d_{p-\mathrm{var}}\big(\bx,\, \Sig{x_n}^{\le N}\big) = 0$,
    \item\label{item:rde_2} $\Gamma_\sigma(y_0; x_n) \xrightarrow[n\to\infty]{} y$ uniformly on $[0,T]$.
\end{enumerate}

In order to state the main existence and uniqueness result for rough differential equation we introduce a convenient Lipschitz space.
Let $V$ and $W$ be two Banach spaces. For $\gamma \in (0,1]$ and a map $f: V \to W$ we define
\begin{align*}
    \Vert f \Vert_{\mathrm{Lip}^\gamma(V;W)} = \max\left\{ \sup_{u \in V}\Vert f(u)\Vert,\;  \sup_{u,v \in V}\frac{\Vert f(v)-f(u)\Vert}{\Vert u-v\Vert^\gamma}\right\}
\end{align*}
and, for $\gamma > 1$ and a  $\integer{\gamma}$-times Fréchet-differentiable map $f: V \to W$, we define recursively
\begin{align*}
    \Vert f \Vert_{\mathrm{Lip}^\gamma(V;W)} = \max\left\{ \sup_{u \in V}\Vert f(u)\Vert,\;  \Vert f^\prime\Vert_{\mathrm{Lip}^{\gamma-1}(V, L(V;W))}\right\},
\end{align*}
where $L(V;W)$ denotes the Banach space of bounded linear functions from $V$ to $W$ (see \cite[Appendix~B]{FV10} for more detail).
In the following we will simply write $\Vert f \Vert_{\mathrm{Lip}^\gamma} = \Vert f \Vert_{\mathrm{Lip}^\gamma(V;W)}$ as the spaces $V$ and $W$ can be inferred from the context.
Furthermore, if we require that $\Vert f \Vert_{\mathrm{Lip}^\gamma} < \infty$ then we implicitly require that $f$ is $\integer{\gamma}$-times differentiable and we set $$\mathrm{Lip}^\gamma(V;W)\coloneqq\{ f:V\to W \;\vert\; \Vert f \Vert_{\mathrm{Lip}^\gamma} < \infty\}$$

Given that $\Vert \sigma \Vert_{\mathrm{Lip}^\gamma} < \infty$ for some $\gamma > p$ it holds that equation \eqref{eq:introductary_rde} has a unique solution $y \in C([0,T];\R^m)$,  cf. \cite[Theorem 10.26]{FV10}.
Moreover, $y$ has a unique lift to a geometric rough path $\by \in \rps(\R^m)$ such that for any sequence of piecewise smooth paths $(x_n)$ satisfying \ref{item:rde_1}  and $y_n :=\Gamma_\sigma(y_0, x_n)$ it holds
\begin{align*}
   \lim_{n\to\infty} d_{p-\mathrm{var}}(\by, \Sig{y_n}^{\le \integer{p}}) = 0.
\end{align*}
Furthermore, the so called Itô-Lyons solution map
\begin{align*}
    \R^m \times \rpsz(\R^d)\times \mathrm{Lip}^\gamma(\R^m; \R^{m\times d}) &\to \rps(\R^m)\\
    (y_0, \bx, \sigma) &\mapsto  \by
\end{align*}
is locally Lipschitz continuous  \cite[Theorem 10.38]{FV10}. We will say that $\by$ solves the \emph{full rough differential equation}
\begin{align*}
    \d{\by}(t) = \sigma(\by(t)) \d{\bx}(t),\quad t \in [0,T], \qquad \by(0) \in G^{\lfloor p \rfloor}(\R^m).
\end{align*}

\section{Approximation with signature controls}\label{sec:continuous_approximation}
We are now ready to provide the main results of this paper, namely that the general, rough and non-Markovian stochastic control problem introduced in Section~\ref{sec:problem_description}, can be solved by controls, which are functions of the signature of the driving path. 
In order to do so, we first need a new, general uniform stability result for rough differential equations whose drift depends on a measurable control, see Section~\ref{sec:drift-controlled-rde}. 
We continue by showing that progressively measurable admissible controls can be approximated by continuous controls, i.e., continuous function on the space of stopped rough paths, see Section~\ref{sec:progressive_approx}.
In Section~\ref{sec:signature_strategies} we derive an appropriate universal approximation theorem (a.k.a.~Stone-Weierstrass theorem), which allows us to conclude that signature controls (i.e., controls defined as linear functionals of the signatures or as neural networks applied to the log-signature) are dense in the set of all admissible controls.
Finally, we prove that we can solve the original stochastic optimal control problem by restricting admissible controls to signature controls, in the sense that the infimum of the expected cost is equal, see Section~\ref{sec:aprx_costs}.

\subsection{Drift-controlled rough differential equations}
\label{sec:drift-controlled-rde}

In the following, we will consider rough differential equations with a controlled drift term
\begin{align}\label{eq:rde_with_drift}
    \d{y}(t) = b(y(t), u(t)) \, \d{t} + \sigma(y(t)) \, \d{\bx}(t),\quad t \in [0,T], \qquad y(0) = y_0\in\R^m,
\end{align}
where $u: [0,T] \to \mathcal
U \subset \R^k$ is measurable and $b:\R^m \times \R^k \to \R^m$.

The smooth approximation concept for solutions of rough differential equations, as introduced in Section~\ref{sec:rde_intro}, naturally extends to equations of the form \eqref{eq:rde_with_drift}.
The following theorem establishes the existence, uniqueness and stability of solutions, which will be crucial for proving our main approximation result in Section~\ref{sec:aprx_costs}.
Substituting $\mathbf{x} = \mathbb{X}(\omega)$ and $u = U(\omega)$, the solution of \eqref{eq:rde_with_drift} corresponds to the solution of the controlled stochastic RDE \eqref{eq:main_srde} for a given realization of the noise as well as the control.

\begin{theorem}\label{thm:diehl_et_al}
        Let $\bx \in \rpsz(\R^d)$ for some $p\in [1, \infty)$ and assume that
    \begin{align}\label{ass:first_on_b}
        \sup_{a\in \mathcal{U}} \Vert b(\cdot, a) \Vert_{\Lip^1} < \infty  \qquad\text{and} \qquad \Vert \sigma \Vert_{\mathrm{Lip}^{\gamma + 1}} < \infty \quad \text{for some } \gamma > p.
    \end{align}
    Then equation \eqref{eq:rde_with_drift} has a unique solution $y$ with corresponding lift $\by \in \rps(\R^m)$ such that
        \begin{align*}
            \R^m \times \rpsz(\R^d) &\to  \rps(\R^m) \\
            (y_0, \bx) &\mapsto  \by
        \end{align*}
        is locally Lipschitz continuous uniformly in $u$.
Assuming further that
        \begin{align}\label{ass:on_b}
            \sup_{a,a^{\prime}\in \mathcal{U}}\dfrac{\Vert b(\cdot, a) - b(\cdot, a^{\prime})\Vert_{\Lip^1}}{\Vert a - a^{\prime}\Vert} <\infty
        \end{align}
        and denoting by $L^0(\d{t}; \mathcal{U}) = \{u:[0,T] \to \mathcal{U} \text{ measurable}\}$ equipped with the topology of convergence in Lebesgue-measure, then also the map
$$\Gamma_{b,\sigma}: \R^m \times \rpsz(\R^d) \times L^0(\d{t}, \mathcal{U}) \to \rps(\R^m), \quad 
    (y_0, \bx, u) \mapsto  \by,$$
    is continuous.
\end{theorem}

\begin{proof}
    The proof uses a similar flow decomposition method  as in \cite{RS17}. 
    
    \noindent Let $\phi \colon [0,T] \times \R^m \to \R^m$ denote the solution to
    \begin{align}\label{eqn:RDE_flow}
        \phi(t,y_0) = y_0 + \int_0^t \sigma(\phi(s,y_0)) \, \d{\bx}(s).
    \end{align}
    From \cite[Proposition 11.11]{FV10}, we know that $\phi$ is a flow of $\mathcal{C}^2$-diffeomorphisms and that the first and second derivate of $\phi$ and its inverse are bounded by a constant depending only on $p,\gamma, \Vert \sigma \Vert_{\mathrm{Lip}^{\gamma + 1}}$ and the rough path norm of $\bx$.  Define
    \begin{align*}
        \mu(t,z) \coloneqq (D_z \phi(t,z))^{-1} b(\phi(t,z),u(t)).
    \end{align*}
    By our assumption on $b$ stated in \eqref{ass:first_on_b} and the properties of $\phi$ we summarized above, $z \mapsto \mu(t,z)$ is continuous for every $t$ and $t \mapsto \mu(t,z)$ is measurable for every $y$. Moreover,
    \begin{align*}
        |\mu(t,z)| \leq \sup_{s \in [0,T]} \| (D \phi(s,\cdot))^{-1}\|_{\infty} \sup_{a \in \mathcal{U}} \sup_{y \in \R^m} |b(y,a)| < \infty
    \end{align*}
    for every $(t,z) \in [0,T] \times \R^m$. Since the inverse of the second derivative of $\phi$ is bounded as well, there exists a constant $C$ such that
    \begin{align*}
        \sup_{t \in [0,T]} |\mu(t,z_1) - \mu(t,z_2)| \leq C |z_1 - z_2|
    \end{align*}
    for every $z_1, z_2 \in \R^m$. Therefore, by Carath\'eodory's theorem, we can conclude that the ordinary differential equation
    \begin{align}\label{eqn:ODE_flow}
        Z^{y_0}_t = y_0 + \int_0^t \mu(s, Z^{y_0}_s) \, \d s
    \end{align}
    has a unique solution $Z^{y_0}$ for every initial condition $y_0$. Define 
    \begin{align*}
        y(t) \coloneqq y^{y_0,\mathbf{x}}(t) \coloneqq \phi(t, Z_t^{y_0}).
    \end{align*}
    
    Now let $\tilde{\bx}$ be another rough path and $\tilde{y}_0$ an initial condition. Choose $\kappa >0$ such that $\kappa \geq d_{p-\text{var}}(\mathbf{1},\bx) \vee d_{p-\text{var}}(\mathbf{1},\tilde{\bx})$. The difference of the corresponding ODE solutions \eqref{eqn:ODE_flow} can be estimated in terms of the difference of the respective flows defined in \eqref{eqn:RDE_flow}. In \cite[Theorem 11.12]{FV10}, it is proven that the flow induced by a rough differential equation is continuous with respect to the driving rough path. In fact, viewing the flow as a solution to a standard, high-dimensional rough differential equation, the local Lipschitz continuity of the It\^o-Lyons map (\cite[Theorem 10.26]{FV10}) can be used to strengthen this result, showing that the flow is even locally Lipschitz continuous with respect to the driving rough path (see the proof of \cite[Corollary 1]{FR11} where this argument is spelled out in more detail). Using this, we can prove an estimate of the form
    \begin{align}\label{eqn:RDE_Lipschitz}
        d_{p-\text{var}}(y^{y_0,\mathbf{x}},\tilde{y}^{\tilde{y}_0,\tilde{\mathbf{x}}}) \leq C(|y_0 - \tilde{y}_0| + d_{p-\text{var}}(\bx,\tilde{\bx}))
    \end{align}
    where the constant $C$ depends on $\kappa$, but not necessarily on the control $u$ due to the fact that the Lipschitz norm of $b$ is uniformly bounded as assumed in \eqref{ass:first_on_b}. 
    
    Note that if $\mathbf{x} = \Sig{x}^{\le \integer{p}}$ for a smooth path $x$, the classical chain rule implies that $y$ solves
    \begin{align*}
        \d{y}(t) = b(y(t), u(t)) \, \d{t} + \sigma(y(t)) \dot{x}(t) \, \d t,\quad t \in [0,T], \qquad y(0) = y_0\in\R^m.
    \end{align*}
    Set $\Gamma_{b,\sigma}(y_0;x) \coloneqq y$. To prove that $y$ solves \eqref{eq:rde_with_drift} in the rough case, too, we choose a sequence of smooth paths $(x_n)_{n\ge1}$ such that 
    \begin{align*}
        \lim_{n\to\infty} d_{p-\mathrm{var}}\big(\bx,\, \Sig{x_n}^{\le \integer{p}}\big) = 0.
    \end{align*}
    The estimate \eqref{eqn:RDE_Lipschitz} implies that indeed 
    \begin{align*}
        \Gamma_{b,\sigma}(y_0;x_n) \to y
    \end{align*}
    as $n \to \infty$ uniformly on $[0,T]$ which proves that $y$ is the unique solution to \eqref{eq:rde_with_drift}. Furthermore, the bound \eqref{eqn:RDE_Lipschitz} implies that 
    \begin{align*}
        (y_0,\mathbf{x}) \mapsto y
    \end{align*}
    is locally Lipschitz continuous uniformly in $u$.

    The rough path lift $\by$ of $y$ is constructed similarly. The only thing that changes is that we have to replace $\phi$ by the flow $\Phi$ that is induced by the full rough differential equation
    \begin{align*}
        \d \mathbf{z}(t) = \sigma(\mathbf{z}(t)) \, \d{\bx}(t).
    \end{align*}
    In fact, it is shown in \cite[Theorem 10.35]{FV10} that solutions to full rough differential equations are solutions to ordinary rough differential equations driven by a different vector field, thus the flow $\Phi$ satisfies exactly the same properties as above. We can hence deduce that
        \begin{align*}
            d_{p-\text{var}}(\by^{y_0,\mathbf{x}},\tilde{\by}^{\tilde{y}_0,\tilde{\mathbf{x}}}) \leq C(|y_0 - \tilde{y}_0| + d_{p-\text{var}}(\bx,\tilde{\bx}))
        \end{align*}
        which proves that $\by$ is indeed the rough path lift of $y$ and that 
        \begin{align*}
            (y_0,\mathbf{x}) \mapsto \by
        \end{align*}
    is also locally Lipschitz continuous uniformly in $u$.
    
    It remains to prove that 
    \begin{align*}
        (y_0, \bx, u) \mapsto  \by
    \end{align*}
    is continuous under assumption \eqref{ass:on_b}. We will argue why the map is continuous in the control $u$ alone first. We fix a rough path $\bx$, an initial condition $y_0$ and consider the corresponding flow $\phi$ defined in \eqref{eqn:RDE_flow}. If $u$ and $\tilde{u}$ are two controls, we define
    \begin{align*}
        \mu(t,z) \coloneqq (D_z \phi(t,z))^{-1} b(\phi(t,z),u(t))
    \end{align*}
    resp.
    \begin{align*}
        \tilde{\mu}(t,z) \coloneqq (D_z \phi(t,z))^{-1} b(\phi(t,z), \tilde{u}(t)).
    \end{align*}
          
    Our assumptions from \eqref{ass:on_b} and the already mentioned properties of the flow implies in particular that for $z_1, z_2 \in \R^m$, we have 
    \begin{align*}
        |\mu(t,z_1) - \tilde{\mu}(t, z_2)| 
        &\le c(|z_1 - z_2| + \min\{\vert u(t) - \tilde{u}(t)\vert, C\}), \qquad t\in[0,T],
    \end{align*}
    for appropriate constants $c,C >0$. One then concludes with Gr\"onwall's inequality that
    the solution of
    \begin{align*}
        \tilde{Z}^{y_0}_t = y_0 + \int_0^t \tilde{\mu}(s, \tilde{Z}^{y_0}_s) \, \d s
    \end{align*}
    converges uniformly to the solution $Z^{y_0}$ of \eqref{eqn:ODE_flow}
    as $\tilde{u} \to u$ with respect to Lebesgue-measure.
    From this, we can conclude that $(y_0, \bx, u) \mapsto  y$
    is continuous in $u$. With the same arguments as before, this statement can easily be generalized to the map  $(y_0, \bx, u) \mapsto  \by$ and, eventually, it can be proven that continuity holds in all parameters using the triangle inequality and the continuity results that we have already proven.   
\end{proof}

\begin{remark}
    In \cite[Theorem 29]{diehl2017stochastic}, a similar statement is formulated in the case of $p \in [2,3)$ under the (slightly) weaker assumption  $\Vert \sigma \Vert_{\mathrm{Lip}^{\gamma}} < \infty$. The reason why we need one additional degree of smoothness here is that we use the flow decomposition method for which also the second derivative of the flow $\phi$ plays a role in the estimates to ensure global existence and uniqueness of \eqref{eqn:ODE_flow}. The proof of \cite[Theorem 29]{diehl2017stochastic} avoids the flow decomposition, but relies on a yet to be developed solution theory for rough differential equations on Banach spaces containing a drift parameter that is only Lipschitz continuous. With this it should be possible to get a version of Theorem \ref{thm:diehl_et_al} that only assumes $\Vert \sigma \Vert_{\mathrm{Lip}^{\gamma}} < \infty$.
\end{remark}

\subsection{Approximation of admissible controls by non-anticipative continuous path functions}\label{sec:progressive_approx}
Throughout this section we will assume that Assumption~\eqref{ass:filtration} holds, i.e., $\X$ takes values in $\rpsz$ and generates the underlying filtration $(\F_t)$.

\begin{proposition}\label{prop:optional_representation}
 For any $(\mathcal{F}_t)$-progressively measurable process $U \colon [0,T] \times \Omega \to \R^k$ there exists a Borel measurable map $\theta \colon \Lambda_T \to \R^k$ such that for
 \begin{align}\label{eq:U-rep}
    \theta(\X(\omega) |_{[0,t]}) = U_t(\omega),
 \end{align}
for $\mathrm{Leb}\otimes\P$-a.e.~$(t,\omega)\in[0,T]\times\Omega$.
 Conversely, for any Borel measurable $\theta \colon \Lambda_T^p \to \R^k$,
 \begin{align}\label{eq:U-theta}
     U^\theta_t : = \theta(\X |_{[0,t]}), \qquad t\in[0,T],
 \end{align}
 defines an $(\F_t)$-progressively measurbale process.
\end{proposition}
\begin{proof}
Let $U$ be a $(\F_t)$-progressively measurable process.
Then $U: [0,T] \times \Omega \to \R^{k}$ is measurable with respect to $\mathcal{B}([0,T])\otimes \F_T$.
From Lemma~\ref{lem:borel_algebra} we have that $\F_T$ is the completion of $\X^{-1}(\mathcal{B}(\rpsz))$.
Hence, there exists a measurable map $\eta: [0,T]\times\rpsz \to \R^{k}$ such that $U = \eta(\cdot, \X)$ upto indistinguishability.

Now fix an arbitrary $t\in[0,T]$.
Since $U_t$ is $\F_t$-measurable it follows again by Lemma~\ref{lem:borel_algebra} that there exists a map $\eta_t: \rpszt \to \R^{k}$ such that almost surely
$\eta(t, \X) = U_t = \eta_t(\X\vert_{[0,t]})$. 
This readily implies that almost surely $\eta(t, \X)  = \eta(t, \X_{\cdot\wedge t})$. Finally we define $\theta: \Lambda_T^p \to \R^{k}$ by setting $\theta(\X\vert_{[0,t]}) := \eta(t, \X_{\cdot\wedge t})$ which satisfies \eqref{eq:U-rep}. From the continuity of the extension map $\Lambda_T \to [0,T] \times \rpsz$, $(t,\bx^t) \coloneqq (t, \bx^t(\cdot \wedge t))$
 (see Section~\ref{sec:stopped_rp}) we also have that $\theta$ is measurable.

Conversely, if $U$ is a process satisfying \eqref{eq:U-theta} for some Borel-measurable map ${\theta: \Lambda_T \to \R^k}$, then for any $t\in[0,T]$ and $B \in \mathcal{B}(\R^k)$ we have
\begin{align*}
    \{ (s,\omega) \in [0, t]\times \Omega \;\vert\; U_s(\omega) \in B\} &= \{ (s,\omega) \in [0, t]\times \Omega \;\vert\; \X(\omega)\vert_{[0,s]} \in \theta^{-1}(B)\} \\
    &= \{ (s,\omega) \in [0, t]\times \Omega \;\vert\; \varphi(s, \X(\omega)\vert_{[0,t]}) \in \theta^{-1}(B)\},
\end{align*}
where $\varphi: [0,t] \times \rpszt \to \Lambda_t: (s, \bx) \mapsto \bx\vert_{[0,s]}$. 
By Proposition~\ref{prop:lambda_topology} the map $\varphi$ is continuous.
Since $\omega \mapsto \X(\omega)\vert_{[0,t]}$ is $\F_t$ measurable it then follows that the above set is in $\mathcal{B}([0,t]) \otimes \mathcal{F}_t$.
\end{proof}

Based on the above result, we are next going to show that admissible controls can be approximated with \emph{continuous} non-anticipative path functionals, i.e., with controls in the following subclass
\begin{align*}
    \mathcal{A}_{c} := \{ U^{\theta} \; \vert\; \theta: \Lambda_T^p \to \R \;\text{ continuous.}\}
\end{align*}

\begin{corollary}\label{cor:continuous_approximation}
For any admissible control $U \in \mathcal{A}$ there exists a sequence $(U^n)$ in $
\mathcal{A}_{c}$ such that
\begin{align}\label{eq:convergence_of_control}
U = \lim_{n\to\infty}U^{n} \qquad \mathrm{Leb}\otimes\P\text{-almost everywhere.}
\end{align}
\end{corollary}

We denote by $P_\mathcal{U}: \R^k \to \mathcal{U}$ the projection map onto the closed convex set $\mathcal{U}$.
Note that by the Hilbert projection theorem (c.f. \cite[Theorem 4.10]{rudin1987complex}) this map is unique and continuous.

\begin{proof}
The continuity of $\varphi: [0,T] \times \rpsz \to \Lambda^p_T:$, $(t, \bx) \mapsto \bx\vert_{[0,t]}$ implies that the push forward $\mu = \varphi(\cdot, \X(\cdot))_\ast[\mathrm{Leb}\otimes\P]$ is a well-defined probability measure on $(\Lambda^p_T, \mathcal{B}(\Lambda^p_T))$. From \cite[p. 148]{Wis94} it follows that every Borel measurable map ${\theta \colon \Lambda^p_T \to \R^k}$ can be approximated $\mu$-almost surely by continuous maps.

Now let $U \in \mathcal{A}$. By Proposition~\ref{prop:optional_representation} there exits a Borel measurable $\theta\colon \Lambda^p_T \to \R^k$ such that \eqref{eq:U-rep} holds for $\textrm{Leb}\otimes \P$-a.e.  $(t,\omega)$.
From the previous paragraph it follows that there exists a sequence of continuous maps $\theta_n: \Lambda^p_T \to \R^k$ for $n\ge1$ that converges $(\mathrm{Leb}\otimes\P)$-almost everywhere towards $\theta$.
This implies that 
$$ U_t(\omega) = \theta(\X\vert_{[0,t]}(\omega)) = \lim_{n\to\infty} \theta_n(\X\vert_{[0,t]}(\omega)),
$$
for $\mathrm{Leb}\otimes\P$-a.e.~$(t,\omega)\in[0,T]\times\Omega$. Finally, composing with the projection map $U^n :=  P_\mathcal{U}\circ\widetilde{\theta}_{n} \circ \X$ and recalling that $P_\mathcal{U}$ is continuous, the claim follows.
\end{proof}

\subsection{Linear and deep signature controls}\label{sec:signature_strategies}

The universal approximation theorem for signatures allows us to approximate any continuous functional on the path space by linear functionals of the signature. 
This motivates us to consider the following class of \emph{linear signature functionals}
\begin{multline*}
\Tsig := \Big\{ \theta: \Lambda^T \to \R^k \;\Big\vert\; \exists \ell_1, \dots, \ell_k \in T((\R^{d+1})^{\ast})
\text{ s.t. } \theta^i(\hat{\bx}^t) = \langle \ell_i, \Sig{\hat{\bx}^{t}}_{0,t}\rangle\;\; \forall\;\bx^{t}\in\Lambda^p_T \Big\}.
\end{multline*}
In general such functionals $\theta \in \Tsig$ will not lead to admissible strategies $U^{\theta}$.
Recall that $P_\mathcal{U}$ is the projection from $\R^k$ onto the convex set $\mathcal{U}$.
We then define the set of \emph{linear signature controls}
\begin{align*}
\Asig := \Big\{ U^{\theta} \;\Big\vert\;  \exists \theta^{\prime} \in \Tsig \text{ s.t. }\theta = P_\mathcal{U} \circ \theta^{\prime} \Big\}.
\end{align*}
We will argue below that $\Tsig \subset C(\Lambda^p_{T}; \R^k)$, thus verifying that indeed $\Asig \subset \mathcal{A}_{c} \subset \mathcal{A}$.
Before doing so we will, however, introduce a second class of signature strategies, based on non-linear functionals.
Motivated by the numerical efficiency for optimal stopping problems in \cite{BHRS23}, we define the following class of deep neural functionals of the log signature
\begin{multline*}
\Tlog := \Big\{ \theta: \Lambda^T \to \R^k \;\Big\vert\;  \exists N\in\N,\; F \in \mathcal{D}^{\eta_{d+1,N}, k} \\ \text{ s.t. } \theta(\hat{\bx}^{t}) = F \circ \log_\otimes (\Sig{\hat{\bx}^{t}}^{\le N}_{0,t}) \;\; \forall\; \bx^{t}\in\Lambda^p_T \Big\},
\end{multline*}
where $\log_\otimes : G^{N}(\R^{d+1}) \to \mathfrak{g}^{N}(\R^{d+1}) \cong \R^{\eta_{d+1,N}}$ is the truncated tensor logarithm (see Section~\ref{sec:tensor}), $\eta_{d+1,N}$ is the dimension of the truncated log signature\footnote{We refer to \cite{BHRS23} Section 7 for more detail on the log-signature and the dimension of the step-$N$ nilpotent free Lie-algebra $\mathfrak{g}^{N}(\R^{d+1})$.
Note also that for simplicity we identify $\mathfrak{g}^N(\R^{d+1})$ with $\R^{\eta_{N}, k}$ in the definition of $\Tlog$ without formally introducing an isomorphism.} and $\mathcal{D}^{l,k}$ is the class of deep neural networks mapping from $\R^l$ to $\R^k$ with fixed depth $I\ge 1$, number of neurons $q \ge 1$ and activation function\footnote{An activation function is a continuous function that is not a polynomial. For example the ReLu function $\varphi(x) = \max\{0,x\}$ (componentwise).} $\varphi: \R^q \to \R^q$.
More precisely, $\mathcal{D}^{l,k}$ consists of functions of the form
\begin{align*}
 F = A_0 \circ \varphi \circ A_1 \circ \varphi \circ \cdots \circ A_I,
\end{align*}
where $A_I:\R^l \to \R^q$, $A_0:\R^q \to \R^k$ and $A_i:\R^q \to \R^q$ $(0 < i < I)$ are affine maps.
Similar to the linear case, we define the set of \textit{deep signature controls} by
\begin{align*}
\Alog := \left\{ U^{\theta} \;\Big\vert\;  \exists \theta^{\prime} \in \Tlog \text{ s.t. }\theta = P_\mathcal{U} \circ \theta^{\prime} \right\}.
\end{align*}
The following verifies that the signature controls are dense in the set of admissible strategies.
\begin{proposition}\label{prop:dens_sig_controls}
    It holds that $\Tsig, \Tlog \subset C(\Lambda^T; \R^k)$.
    Furthermore, given assumption \eqref{ass:filtration}, it holds that $\Asig, \Alog \subset \mathcal{A}_c$ and for any admissible control $U \in \mathcal{A}$ there exists a sequence $(U^n)_{n\ge1} \subset \Asig$ (respectively in  $\Alog$) such that
    \begin{align*}
 U = \lim_{n\to\infty}U^{n}, \qquad \mathrm{Leb}\otimes\P\text{-almost everywhere.}
\end{align*}
\end{proposition}

Based on Corollary~\ref{cor:continuous_approximation} the proof is similar to the  Proposition~7.4 in \cite{BHRS23}.

\begin{proof}
Recall from Lemma~\ref{lem:continuity_timesig} that the map $\bx^t \mapsto \Sig{\hat\bx^t}^{\le N}_{0,T}$ is continuous for any $N\in\mathbb{N}$.
Since also the linear map $\langle \ell, \cdot \rangle: T((\R^{d+1})) \to \R$ is continuous for any $\ell \in T(\R^{d+1})$ we readily conclude that $\Tsig \subset C(\Lambda_T; \R^k)$. 
Similarly, since any $F\in\mathcal{D}^{\eta_N, k}$ is a continuous function and recalling from Section~\ref{sec:tensor} that also the truncated logarithm $\log_\otimes: G^N(\R^{d+1}) \to \mathfrak{g}^N(\R^{d+1})$ is continuous it follows similarly that $\Tlog \subset C(\Lambda_T; \R^k)$.
The fact that $\Asig, \Alog \subset \mathcal{A}_c$ then follows immediately by the continuity of the projection $P_\mathcal{U}$.

To prove the approximating result it suffices by Corollary~\ref{cor:continuous_approximation} to show that for any $\theta \in C(\Lambda_T;\R^k)$ there exits a sequence $(\theta_n)_{n\ge1} \subset \Tsig$ (respectively $\Tlog$), such that for a.e. $\omega\in\Omega$ it holds
\begin{align}\label{eq:_proof_approx}
     \theta(\X\vert_{[0,t]}(\omega)) = \lim_{n\to\infty} \theta_n(\X\vert_{[0,t]}(\omega)),\qquad \text{for }\mathrm{Leb}\text{-a.e. } t\in[0,T].
\end{align}
To this end, since $\rpsz$ is a Polish space, we can choose an increasing sequence of compact sets $K_n \subset \rpsz$ such that $\lim_{n\to\infty}\P(\X \in K_n) = 1$.
By Proposition~\ref{prop:universial_signature_stopped} for each $n\in\mathbb{N}$ there exits a sequence $(\ell_{n,j})_{j\ge1}\subset T(\R^{d+1})^{\otimes k}$ such that for $j_n \in \N$ large enough it holds 
$$ \sup_{t\in[0,T]} \sup_{\bx \in K_n} \max_{i = 1, \dots, k}\Big( \big\vert \theta^{i}(\bx\vert_{[0,t]}) - \langle \ell^{i}_{j}, \Sig{\hat{\bx}\vert_{[0,t]}} \rangle \big\vert \Big)  \le \frac{1}{n}, \qquad j \ge j_n.$$
Defining the sequence $(\theta_n)_{n\ge1}$ by $\theta^i_n(\bx^t) := \langle \ell^{i}_{j_n}, \Sig{\hat{\bx}^t}_{0,t}\rangle$ for all $i=1, \dots, k$, we readily see that it satisfies \eqref{eq:_proof_approx}.
To prove the approximation result for $\Alog$ it now suffices to show that $\Tlog$ is suitably dense in $\Tsig$.
More precisely, it suffices to show that for any compact set $K \subset \rpsz$, $N\in \mathbb{N}$ and  $\ell \in T^N(\R^{d+1})$ there exists a sequence of functions $(F^n)_{n\ge 1} \in \mathcal{D}^{\eta_{d+1,N},1}$ such that
\begin{align*}
    \lim_{n\to\infty}\sup_{\bx\in K}\vert \langle \ell, \Sig{\hat{\bx}}_{0,T}\rangle - F^n(\log_\otimes \Sig{\hat{\bx}}_{0,T})\vert = 0.
\end{align*}
Recall from Section~\ref{sec:tensor} that the map $\log_\otimes: G^N(\R^{d+1}) \to \mathfrak{g}^N(\R^{d+1})$ is continuous and invertable with inverse $\exp_{\otimes}$.
Hence, in particular the set $\{\log_\otimes \Sig{\hat{\bx}}_{0,T} \;\vert\; \bx \in K\} \subset \mathfrak{g}^N(\R^{d+1})$ is compact.
Since also the map $\mathfrak{g}^N(\R^{d+1}) \to \R: \mathbf{z} \mapsto \langle \ell, \exp_\otimes(\mathbf{z})\rangle$ is continuous, the statement now readily follows from the universal approximation theorem for nerual networks (see e.g. \cite[Theorem 1]{leshno1992multilayer}).
\end{proof}

\subsection{Approximation of the optimal costs}\label{sec:aprx_costs}

In this section we present our main theoretical result, which states that the optimal expected costs of the problem introduced in Section~\ref{sec:problem_description} can be approximated using the classes of signature controls $\Asig$ and $\Alog$.
We begin by proving the continuous dependence of the costs with respect to the underlying control.
Given the continuity of the cost functional $L$, this is a direct consequence of the stability result in Theorem~\ref{thm:diehl_et_al}.

\begin{lemma}\label{lem:continuity-Y-controlled}
Assume that \eqref{ass:cost_functional}, \eqref{ass:first_on_b} and \eqref{ass:on_b} hold and let $(U^{n})_{n\ge1}\subset \mathcal{A}$ and $U\in \mathcal{A}$ s.t.
\begin{align*}\
U = \lim_{n\to\infty} U^{n},\qquad \mathrm{Leb}\otimes\P\text{-almost everywhere.}
\end{align*}
Denote by $\Y^{n} = \Y^{U^{n}}$ respectively $\Y = \Y^{U}$ the solution to the rough differential equation \eqref{eq:main_srde} with the control $U^{n}$ respectively $U$.
Then it holds
\begin{align*}
L(\Y, U) = \lim_{n\to\infty}L(\Y^{n}, U^{n}), \qquad \text{almost surely}.
\end{align*}
\end{lemma}
\begin{proof}
We have in particular that $U^n(\omega)$ converges towards $U(\omega)$ in $L^0(\d{t},\mathcal{U})$ for a.e.~$\omega\in\Omega$. From Theorem~\ref{thm:diehl_et_al} it follows that
\begin{align*}
\lim_{n\to\infty} d_{p-\textrm{var}}(\Y^{n}(\omega), \Y(\omega)) = 0, \qquad \text{for a.e. } \omega \in \Omega.
\end{align*}
We readily conclude using the continuity assumptions on $L$.
\end{proof}

Combining Corollary~\ref{cor:continuous_approximation} and Proposition~\ref{prop:dens_sig_controls} with the above lemma, the following result is now a direct consequence of the dominated convergence theorem.
\begin{theorem}\label{thm:approximation}
Given assumptions \eqref{ass:cost_functional}, \eqref{ass:filtration}, \eqref{ass:first_on_b} and \eqref{ass:on_b} we have
\begin{align*}
\inf_{U\in\mathcal{A}}J(U) ~=~ \inf_{U\in\mathcal{A}_{c}}J(U) ~=~ \inf_{U \in \Asig}J(U) ~=~ \inf_{U \in \Alog}J(U).
\end{align*}
\end{theorem}

\section{Numerical Method and Examples}\label{sec:numerics} 
In this section, we will introduce a numerical method that arises from the parametrization of admissible controls in Section 4.2. We will then evaluate its performance through two case studies of non-Markovian control problems.
The full code is available at \url{https://github.com/hagerpa/sigControl}.

\subsection{Numerical Method}\label{sec:numerical_method}
Fixing a signature truncation level $N\ge 1$, the corresponding sub-classes $\Asig^N \subset \Asig$ and $\Alog^N \subset \Alog$ constitute finite-dimensional parameterizations of admissible controls. 
Indeed, for $\Asig^N$ the set of parameters is given by the coefficients in the truncated tensor series $\ell_1, \dots, \ell_k \in T^N((\R^{d+1})^{\ast})$, i.e., the dimension of the parameter space is given by $k \cdot (1 + (d+1) + \cdots (d+1)^N)$.
For $\Alog^N$, the parameter space is determined by the architecture of the deep neural networks $\mathcal{D}^{\eta_{d+1,N}, k}$, thus characterized by the number of hidden layers $I$ and the number of neurons per layer $q$.
Let us reveal the parametrization in both cases more explicitly by noting that for any $U \in \Asig^N$ (resp. $\Alog^N$) there exists a vector of parameters $\eta$ that defines a function $\theta(\cdot\;; \eta) : G^N(\R^{d+1}) \to \mathcal{U}$ such that $U_t = \theta(\Sig{\hat\X}^{\le N}_{0,t}; \eta)$.
For ease of notation we then also write $U(\eta) \in \Asig^N$ (resp. $\Alog^N$).

The next best objective towards solving the control problem is then to minimize $J(U(\eta))$ over  $\eta$.
This is achieved numerically through a Monte Carlo approximation of the expectation and a time discretization for approximating the signature and solving the rough differential equation.
We will first present the general structure of the algorithm and then comment on the details below.
For the sake of concreteness, we will assume that the cost functional $L$ is of the integral form\footnote{In general, the discretization of the cost functional depends on its specific form, while under the given assumption, a straightforward discretization is given by the Riemann sum.} \eqref{eq:running_costs}.

  \begin{enumerate}[label={(\arabic*)}]
  \item Fix a time grid $\Pi = \{ t_0, t_1, \dots, t_n\}$ with $0 = t_0 < t_1 < \dots < t_n = T$.
   \item \label{enum:alg_sig} Generate $M$ independent realizations $\{(\delta\mathbf{S}^{(i)}_j)_{j=1, \dots, n} \;\vert\; i = 1, \dots, M\}$ of $(\Sig{\hat\X}^{\le N}_{t_{j-1}, t_j})_{t_j \in \Pi}$; then set $ \mathbf{S}^{(i)}_0 = \mathbf{1}$ and iteratively $\mathbf{S}^{(i)}_j = \mathbf{S}^{(i)}_{j-1}\otimes \delta \mathbf{S}^{(i)}_j $.
   
   \item \label{enum:alg_cont} Evaluate the signature controls $U^{(i);\eta}_j = \theta(\mathbf{S}^{(i)}_j; \eta)$ and store the gradients $\nabla_\eta \theta(\mathbf{S}^{(i)}_j; \eta)$, $j=1,\dots,n$.

   \item Calculate or use a suitable numerical scheme to find (approximate) solutions $\{(Y^{(i);\eta}_j)_{j=0, \dots, n} \;\vert\; i = 1, \dots, M\}$ on the grid $\Pi$ of the rough differential equation \eqref{eq:vol_controlled} corresponding to the samples $(\delta\mathbf{S}^{(i)}_j)$.

   \item \label{enum:alg_cost}Calculate the approximate expected costs
   \begin{align}\label{eqn:MC_approx_optimal_value}
       \frac{1}{M}\sum_{i=1}^M \left\{\sum_{t_j \in \Pi}f(t_j, Y^{(i);\eta}_j, U^{(i);\eta}_j)(t_j - t_{j-1}) + g(Y^{(i);\eta}_{n}) \right\}
   \end{align}
   and its gradients with respect to $\eta$ using the previously stored gradients of the control. Then updated $\eta$ using a step of stochastic gradient descent or a similar method.
   \item Starting from 2., repeat the above procedure with the new parameters several times or until no further improvement.
   \item For a larger number of samples $M^\prime \gg M$ follow the steps 2. - 5. above once again and use \eqref{eqn:MC_approx_optimal_value} as an approximation to the optimal control value.
  \end{enumerate}

\subsection*{Calculation of the signature}
In most cases of interest, we are confined to approximate realizations of the signature $\Sig{\hat\X}^{\le N}_{t_{j-1}, t_j}$ in step \ref{enum:alg_sig} since direct sampling is not possible. Assume that $\X$ is the limit of the lifted piecewise linear approximation of the process $X = \pi_1(\X)$, from which we can generate samples\footnote{If direct sampling from the rough path $\X$ is possible, we use piecewise geodesic interpolation instead. However, it's worth noting that even for a Brownian motion, direct sampling from the rough path is not possible.}. As discussed in Section~\ref{sec:rough_paths}, this holds true, for example, in the case of the Stratonovich-lift of a Brownian motion. We can then approximate realizations of the signature by sampling $X$ on a refined grid of the interval $[t_{j-1}, t_j]$ and calculating the signature of the linearly interpolated path. 
Here a refinement can be beneficial for the optimization procedure, since further information from the evolution of $X$ between $t_{j-1}$ and $t_j$ can be incorporated without the need to re-evaluate the control.
The signature of a piecewise linear path can be calculated exactly, and implementations of the underlying algebraic structure are readily accessible, e.g., in the \emph{iisignature} package \cite{iisignature}. This package also provides the necessary functionality to join the signatures over consecutive time intervals. We also note that the log-signature of a piecewise linear path can be calculated directly using the Baker–Campbell–Hausdorff formula, and the iisignature package also provides this functionality.

\subsection*{Numerical schemes}\label{sec:higher-order-scheme}
Similarly to the signature, calculating the RDE solutions %
explicitly is mostly not an option.
Instead, one can use various numerical schemes to calculate approximate solutions.
Having already calculated approximations of the signature increments $(\delta\mathbf{S}^{(i)}_j)$ in the previous step, it is natural to use a higher-order Euler (or Taylor-type) scheme given by
\begin{align}\label{eq:higher-order-scheme}
   Y^{(i);\eta}_j = Y^{(i);\eta}_{j-1} + b(Y^{(i);\eta}_{j-1}, U^{(i);\eta}_{j-1})(t_j - t_{j-1}) + \mathcal{E}_{(\sigma)}(Y_{j-1}^{(i);\eta}, \delta \mathbf{S}^{(i)}_j),
   \end{align}
   for all $j= 1, \dots, n$, starting with $Y^{(i);\eta}_0 = y_0$, where
   \begin{align*}
       \mathcal{E}_{(\sigma)}(y, \delta \mathbf{S}^{(i)}_j) &=  \sum_{k = 1}^{N} \sum_{i_1,\dots, i_k = 1}^d \sigma_{[i_1]}(y) \cdots \sigma_{[i_k]}(y) \mathrm{Id}(y)\langle e_{i_1 \dots i_k},\delta\mathbf{S}_j\rangle \\
       &=
       \sigma(Y_{j-1}^{(i);\eta})\pi_1(\delta\mathbf{S}_{j}) + \sum_{i_1,i_2=1}^d\sum_{l=1}^m \Big[ \sigma_{i_1}^l\,\frac{\partial}{\partial_{y_l}} \sigma_{i_2}\Big](Y_{j-1}^{(i);\eta})\langle e_{i_1 i_2},\delta\mathbf{S}_j\rangle + \dots
   \end{align*}
   with $\sigma_{[i]} := \sum_{l=1}^m \sigma^l_i\frac{\partial}{\partial y_l}$ and the identity map $\mathrm{Id}: y\mapsto y$. Note that one recovers the usual Euler scheme for $N = 1$ and that the choice $N = 2$ corresponds to the Milstein scheme. In general, to guarantee convergence, one needs to choose $N \geq \integer{p}$. This class of schemes were analyzed in \cite{FV10} first. 
   If no refinement of the intervals $[t_{j-1}, t_j]$ is used when calculating the signature increments, we have $\delta \mathbf{S}^{(i)}_j = \exp_{\otimes}^{N}(X^{(i)}_{t_j} - X^{(i)}_{t_{j-1}})$.
   Such \textit{simplified} higher-order Euler schemes and their convergence rates were analyzed e.g. in \cite{DNT12,FR14,BFRS16}. In general, higher-order Euler schemes are easy to implement, but they require the calculation of higher order derivatives, too, which can be costly in high dimensions. Moreover, they have a bad performance in case the underlying RDE is stiff, in which case implicit schemes typically perform better. In \cite{RR22}, a class of simplified Runge-Kutta schemes for RDEs was introduced and analyzed that contains derivative-free and implementable schemes including implicit ones.

\subsection*{Numerical optimization procedure}
The calculation of gradients in step \ref{enum:alg_cont} and \ref{enum:alg_cost} is automatized  when using standard machine learning software libraries such as \emph{PyTorch} \cite{pytorch}.
Such packages also provide state of the art variants of stochastic gradient descent.
In the numerical examples below the performance was rather insensitive to the specific choice of optimization method and the choice of hyper-parameters.
In this case the ``Adam'' method \cite{kingma2014adam} lead to sufficient accuracy.

\subsection*{Linearization of the control problem}
In special cases the algebraic properties of the signature can further be employed to transform the optimization problem into a form that allows to efficiently use non-stochastic solvers.
This approach was first suggested in \cite{kalsi2020optimal} in the context of an optimal execution problem (see also Section~\ref{sec:optimal_execution} below). 
Without going into details, this approach is generally possible if the  the system \eqref{eq:main_srde} can be solved by integration, i.e., when $b(y,u)$ and $\sigma(y)$ are independent of $y$, and when $b$ and $L$ are of a polynomial form.
Given sufficient integrability of $\Sig{\hat\X}_{0, T}$, the shuffle property of the signature then allows to rewrite the expected costs associated to a signature control $U^\ell = \langle \ell, \Sig{\hat\X}_{0, \cdot} \rangle$ for some $\ell \in T^{N}(\R^{d+1})$ as
\begin{align*}
    J(U^\ell) = \langle P(\ell),\E[ \Sig{\hat\X}^{\le N^{\prime}}_{0, T}]\rangle,
\end{align*}
where $P: T^{N}(\R^{d+1}) \to T^{N^\prime}(\R^{d+1})$ with $N^{\prime} \ge N$ is polynomial in the coefficients of ${\ell \in T^{N}(\R^{d+1})}$.
Estimating the truncated expected signature using a Monte-Carlo average an optimal $\ell$ can then be obtain from a deterministic solver for polynomial optimization.
As observed in \cite{kalsi2020optimal}, this reformulation of the control problem preserves quadratic convex structures, and thus allows to solve problems admitting such a structure as quadratic programs.

\subsection{Case study 1: Optimal tracking of fractional Brownian motion}\label{sec:optimal_tracking}

As a first benchmark example we consider the problem of optimally
tracking a fractional Brownian motion with a process whose speed can
be controlled at quadratic costs. Letting the Hurst parameter $H$ vary
in $(0,1]$ allows us to test our numerical method in a range of cases
outside the Markov and semimartingale regimes typically considered in
the literature.

To set the stage, let $\xi$ be a one-dimensional fractional Brownian
motion on some probability space $(\Omega,\mathcal{F},\P)$ and put $X:=\xi$ so
that $\xi$ is trivially adapted to the augmented filtration generated
by $(\sigma(X_s; s \in [0,t]))_{t \geq 0}$. Notice that this way the
controller will \emph{not} have access to the full past before time 0
of the fractional Brownian motion; granting this extended access leads
to a different, yet equally relevant control problem, which, however,
is computationally more demanding and thus left for future experiments
at this point.

Now define the controlled process
\begin{align}\label{eq:trackingdynamics}
Y^{U}_t = y_0 + \int_0^{t}U_s \d{s}-\xi_t,
\end{align} 
where $U$ is a progressively measurable process that represents the
speed and direction of the tracking. Since a one-dimensional
fractional Brownian motion allows for a rough path lift
\begin{align*}
  \mathbb{X}_t = \exp_{\otimes}^{\lfloor p\rfloor}(X_t), \quad \text{
  for some } p\in(1/H, 1+1/H),
\end{align*}
this setting fits into our theoretical framework~\eqref{eq:main_srde}
by choosing $\mathcal{U} = \R$, ${b(y, u) = u}$, $\sigma =-1$.  %
We measure the tracking performance by the cost functional
\begin{align}\label{eq:tracking_problem}
L(\Y^U,U)= \frac{1}{2}\int_0^{T}\left((Y^{U}_t )^{2} + \kappa (U_t)^{2} \right) \d{t}
\end{align}
where  $\kappa > 0$ is the penalization parameter; we also confine
ourselves to the natural class of progressively measurable control processes $U \in
L^2(\mathbb{P}\otimes dt)$.

\begin{remark}
  This kind of tracking problem finds several applications in  mathematical finance.
For instance, following \cite{bank2017hedging}, we can consider $\xi$
to be the hedging strategy of a contingent claim in an idealized frictionless
reference market model driven by $X$; $y_0+\int_0^. U_s \d{s}$ could be the
evolution of a trader's actual hedging position when she is confronted with market impact costs as captured by $\kappa \int_0^T U_t^2\d{t}$.
Due to these costs, the trader has to allow for a nonzero hedging
error $Y^U$ and the resulting risk can be measured by $\mathbb{E}\left[\int_0^{T}\left(\frac{1}{2} (Y^{U}_t )^{2} \right) \d{t}\right]$. Combining this with expected impact costs leads to~\eqref{eq:tracking_problem}.
\end{remark}

Our choice of this particular benchmark problem is motivated by its
analytic tractability: \cite{bank2017hedging}  describes control
policies and the problem value even for general targets $\xi$ in. For our particular
tracking problem with fractional Brownian motion, this gives:

\begin{theorem}\label{thm:fbm_benchmarks}
    The minimal tracking costs for a fractional Brownian motion $\xi$ with Hurst parameter $H \in (0,1)$ are 
    \begin{align} 
      \inf_{U} \E[L(Y^U,U)] = 
      & \; \frac{1}{2} \sqrt{\kappa} \tanh(\tau^\kappa(0)) \left( y_0  \right)^2 \nonumber \\ & +
        \frac{1}{2} \int_0^T \int_0^t\left(\int_s^T(z_H(t,s)-z_H(u,s))\frac{\cosh(\tau^\kappa(u))}
      {\sqrt{\kappa}\sinh(\tau^\kappa(t))}du\right)^2dsdt \label{eq:cost1} \\
      & + \frac{1}{2} \int_0^T \sqrt{\kappa}
        \tanh(\tau^\kappa(t)) \frac{\left(\int_t^Tz_H(u,t)\cosh(\tau^\kappa(u))du\right)^2}{\kappa
      \sinh^2(\tau^\kappa(t))} dt
              < \infty. 
    \nonumber
    \end{align}
    where, for $0 \leq s \leq t \leq T$, we let 
    $
   \tau^\kappa(t) := ({T-t})/{\sqrt{\kappa}} , \quad 0 \leq t \leq T,
$
 and
        \begin{align*}
      z_H(t,s) := c_H 
         \left(\frac{t}{s}\right)^{H-\frac{1}{2}}(t-s)^{H-1/2}-(H-\frac{1}{2})s^{\frac{1}{2}-H}\int_s^tu^{H-\frac{3}{2}}(u-s)^{H-\frac{1}{2}}du,
    \end{align*}
    with
    \begin{align*}
         c_H :=
      \left(\frac{2H\Gamma(\frac{3}{2}-H)}{\Gamma(H+\frac{1}{2})\Gamma(2-2H)}\right)^{\frac{1}{2}}.
    \end{align*}
\end{theorem}

\begin{proof}
  The general form of the optimal tracking strategy is given in Theorem~1 of \cite{bank2017hedging}. In conjunction with this, Theorem~3.4 of~\cite{BankVoss}
 identifies the minimal tracking error as
 \begin{align*} 
      \inf_{U} \E[L(Y^U,U)] = & \;\frac{1}{2}\sqrt{\kappa}\tanh(\tau^\kappa(0))(x-\hat{\xi}_0)^2%
      +
        \frac{1}{2} \mathbb{E} \left[ \int_0^T (\xi_t - \hat{\xi}_t)^2 dt \right] \nonumber \\
      & + \frac{1}{2} \mathbb{E} \left[ \int_0^T \sqrt{\kappa}
        \tanh(\tau^\kappa(t)) d\langle \hat{\xi} \rangle_t \right]
 \end{align*}
 The tracking process $\hat{\xi}$ introduced there takes here the form
     \begin{equation*}
     \hat{\xi}_t :=
     \int_t^T \mathbb{E}\left[\xi_u\middle| \mathcal F_t\right] \frac{\cosh(\tau^\kappa(u))}
      {\sqrt{\kappa}\sinh(\tau^\kappa(t))} du, 
    \quad  0 \leq t \leq T.  
    \end{equation*}
     For $\mathcal F_t$ generated by $\sigma(\xi_s, s \in [0,t])$, the conditional expectation $\mathbb{E}\left[\xi_u\middle| \mathcal F_t\right]$ is computed in Theorem~4.2 in \cite{PoorFBM}. With $z_H$ as above we obtain
 \begin{align*}
     \mathbb{E}\left[\xi_u\middle| \mathcal F_t\right] = \int_0^t z_H(u,s) dB_s
    \end{align*} 
    where \begin{align*}
      B_t := \frac{2H}{c_H}\int_0^t s^{H-\frac{1}{2}}dM^H_s ,
      \quad t \geq 0
 \end{align*}
    is the standard Brownian motion constructed  from the Gaussian martingale obtained from the fBM $\xi$ via
$%
      M^H_t := \int_0^t w_H(t,s) d\xi_s$,  $t \geq 0,
$ %
    with     
    \begin{align*}
w_H(t,s):=\frac{s^{\frac{1}{2}-H}(t-s)^{\frac{1}{2}-H}}{2H\int_0^1u^{\frac{3}{2}-H}(1-u)^{\frac{1}{2}-H}du}
      , \quad s<t.
    \end{align*}

  Straightforward computations starting from this show that
      \begin{align*}
      \mathbb{E} \left[ \int_0^T (\xi_t - \hat{\xi}_t)^2 dt \right]=\int_0^T \int_0^t\left(\int_s^T(z_H(t,s)-z_H(u,s))K(t,u)du\right)^2dsdt
    \end{align*}
    and
    \begin{align*}
      \langle \hat{\xi}\rangle_t = \int_0^t \frac{\left(\int_s^Tz_H(u,s)\cosh(\tau^\kappa(u))du\right)^2}{\kappa
      \sinh^2(\tau^\kappa(s))} ds, \quad t \in [0,T].
    \end{align*}
  Our formula for the minimal tracking error follows now by plugging the identities from the last two displays into the general cost formula stated above.
  \end{proof}

\begin{table}[h]
\centering
\begin{tabular}{ c  r | c c c c c c c c c c c }
  &  $H$ &  1/16 &  1/8 &  1/4 &  1/2 &  3/4 &  1  \\
\hline
  \multicolumn{2}{c |}{th. optimum}   & 0.293 & 0.264 & 0.206 & 0.124 & 0.071 & 0.034  \\ 
\hline
\multirow{ 5 }{*}{ $\Asig$ }& $N = 1$  & 0.329 & 0.286 & 0.223 & 0.146 & 0.101 & 0.073  \\ 
& $N = 2$  & 0.315 & 0.275 & 0.211 & 0.127 & 0.076 & 0.041  \\ 
& $N = 3$  & 0.310 & 0.273 & 0.210 & \textbf{ 0.124 } &  \textbf{0.073} & 0.038  \\ 
& $N = 4$  & 0.304 & 0.270 &  \textbf{0.209} & \textbf{ 0.124 } &  \textbf{0.073} & 0.038  \\ 
& $N = 5$  & 0.305 & 0.270 &  \textbf{0.209} & \textbf{ 0.124 } &  \textbf{0.073} & 0.038  \\ 
\hline
\multirow{ 5 }{*}{ $\Alog$ }& $N = 1$  & 0.315 & 0.276 & 0.214 & 0.135 & 0.083 & \textbf{ 0.034 }  \\ 
& $N = 2$  & 0.307 & 0.272 & 0.210 & \textbf{ 0.124 } &  \textbf{0.073} & \textbf{ 0.034 }  \\ 
& $N = 3$  & 0.301 & 0.269 &  \textbf{0.209} & \textbf{ 0.124 } & \textbf{ 0.072 } & \textbf{ 0.034 }  \\ 
& $N = 4$  & \textbf{ 0.300 } & \textbf{ 0.267 } & \textbf{ 0.208 } & \textbf{ 0.124 } & \textbf{ 0.072 } & \textbf{ 0.034 }  \\ 
& $N = 5$  & \textbf{ 0.300 } & \textbf{ 0.267 } &  \textbf{0.209} & \textbf{ 0.124 } & \textbf{ 0.072 } &  \textbf{0.035}  \\ 
\hline\hline
\multirow{ 5 }{*}{ $\Asig^Y$ }& $N = 1$  & 0.335 & 0.287 & 0.216 & 0.130 & 0.082 & 0.050  \\ 
& $N = 2$  & 0.317 & 0.277 & 0.211 &  \textbf{0.125} & \textbf{ 0.072 } &  \textbf{0.035}  \\ 
& $N = 3$  & 0.309 & 0.272 &  \textbf{0.209} & \textbf{ 0.124 } & \textbf{ 0.072 } &  \textbf{0.035}  \\ 
& $N = 4$  & 0.304 & 0.270 & \textbf{ 0.208 } & \textbf{ 0.124 } & \textbf{ 0.072 } &  \textbf{0.035}  \\ 
& $N = 5$  & 0.304 & 0.269 & \textbf{ 0.208 } & \textbf{ 0.124 } & \textbf{ 0.072 } & 0.036  \\ 
\hline
\multirow{ 5 }{*}{ $\Alog^Y$ }& $N = 1$  & 0.327 & 0.281 & 0.211 & \textbf{ 0.124 } & \textbf{ 0.072 } & \textbf{ 0.034 }  \\ 
& $N = 2$  & 0.304 & 0.271 &  \textbf{0.209} & \textbf{ 0.124 } & \textbf{ 0.072 } & \textbf{ 0.034 }  \\ 
& $N = 3$  & 0.302 & 0.269 &  \textbf{0.209} & \textbf{ 0.124 } & \textbf{ 0.072 } & \textbf{ 0.034 }  \\ 
& $N = 4$  &  \textbf{0.300} & \textbf{ 0.267 } & \textbf{ 0.208 } & \textbf{ 0.124 } & \textbf{ 0.072 } & \textbf{ 0.034 }  \\ 
& $N = 5$  & \textbf{ 0.299 } &  \textbf{0.268} & \textbf{ 0.208 } & \textbf{ 0.124 } & \textbf{ 0.072 } & \textbf{ 0.034 } 
\end{tabular}
\caption{Numerical results for the optimal tracking of a fractional Brownian motion with different Hurst parameters $H$ using (open and  closed loop) controls in $\Asig$ and $\Alog$ ($I = 2$, $q = 30 + \eta_{2, N}$), with various signature truncation levels $N$. Presented are the estimated expected costs. The fixed model parameters are $y_0=0$, $T=1$, and $\kappa=0.1$. An overall time discretization of $\Delta t = 10^{-3}$ was used for the calculation of signatures, $Y^U$, and the cost functional. The number of training and testing paths was $2^{19}$ and $2^{20}$, respectively. The Monte Carlo resampling error was below $0.0002$. The first row presents the continuous-time optimal values calculated using \eqref{eq:cost1}.\label{tab:tracking_problem}}
\end{table}

We have tested our approximation method against the optimal values obtained from Theorem~\ref{thm:fbm_benchmarks} for several choices of Hurst parameters $H$ using strategies from $\Asig$ and $\Alog$ with different signature truncation levels $N$.
The remaining parameters were fixed to $y_0 = 0$, $T=1$  and $\kappa = 0.1$.
The outcomes of these numerical test are collected in Table~\ref{tab:tracking_problem}.
We notice that using our method we can reach reasonably accurate approximations of the continuous time optimum.
Overall the performance is improved when increasing the truncation level and when going from linear strategies in $\Asig$ to deep strategies in $\Alog$.
It turns out that already beyond truncation levels $N=3$ or $N=4$ there is no further improvement and for cases $H\ge 1/2$ even level $N=2$ seems to be sufficient.

In the lower half of Table~\ref{tab:tracking_problem} we present numerical results obtained with a ``closed loop'' version of our method.
Instead of using the signature of the driving noise, the controls denoted by $\Asig^{Y}$ (resp.\ $\Alog^{Y}$) were defined by applying linear functionals (resp.\ DNNs) to the truncated signature of the time-augmented controlled process $\widehat{Y}^U$. This approach is a special case of the method proposed in \cite{hoglund2023a} and we stress that the results are only presented for numerical comparison. Our theory does not cover this type of signature controls.
For more details on the implementation, we refer to the code repository \url{https://github.com/hagerpa/sigControl}.

Clearly, the closed loop version of the problem yields values similar to the open loop formulation studied rigoroulsy in our paper. This closeness can be expected here since in this special model closed loop controls can draw on the same information flow as the open loop controls. Indeed, one can readily reconstruct via~\eqref{eq:trackingdynamics} the driver $\xi$ from the system's evolution $Y^U$ generated by a control $U$ adapted to the filtration generated by the system itself. This makes this particular closed loop problem actually a full information problem with a ``sufficient statistic'' $Y^U$ which differs from the rough driver $\xi$ only by a controlled drift process.

Let us comment on a few more aspects  in detail:

Due to the Markovianity in the Brownian monition case $H=1/2$, there exists an optimal control in feedback form $U^{\ast}_t = \alpha(t, X_t, Y^{U^{\ast}}_t)$ for a suitable deterministic function $\alpha$ (cf. \cite{bank2017hedging}). It turns out that to learn this control in open loop form $U^{\ast}_t = \theta(X\vert_{[0,t]})$ we only need level $N=2$ log-signatures in $\Alog$ and level $N=3$ signatures in $\Asig$ to achieve sufficient accuracy.
The second special case is $H=1.0$, where the trajectories of $X$ are straight lines starting at the origin with a standard normal distributed slope $X_1$ which becomes known right after the start.
In this case the optimal solution is of the form $U^\ast_t = \mathbf{1}_{\{t>0\}}X_1 f(t) = \mathbf{1}_{\{t>0\}} \frac{X_t}{t}f(t)$, where $f: [0,T] \to \R$ is a suitable nonlinear deterministic function (cf. \cite{bank2017hedging}).
While with deep neural networks this control can easily be learned as a function of the first level signature $(t, X_t)$, the linear model needs a much higher truncation level to approximate the nonlinear function $f(t)/t$ with sufficient precision.
This carries over to an explanation for the general outperforming of $\Alog$ strategies over those from $\Asig$ when comparing identical signature truncation levels.

The method turned out to be efficient for this case study despite the fact that the cost functional and the drift coefficient are unbounded over the the set of admissible controls, contrary to assumptions \eqref{ass:cost_functional} and \eqref{ass:first_on_b} needed for the applicability of Theorem~\ref{thm:approximation}.
This carries little surprise as we expect that our approximation results carries over to the case of $L^2$-type cost functionals (see paragraph on the \textit{boundedness of the cost functional} in Section~\ref{sec:problem_description}).

\subsection{Case study 2: Non-Markovian optimal execution}\label{sec:optimal_execution}

As a second benchmark, we consider an instance of the optimal order execution problem of \cite{kalsi2020optimal}. There $X$ describes the fluctuations in the fundamental price of some financial asset. The controller seeks to unwind an initial inventory $q_0>0$ of shares over a (typically short) period $[0,T]$ by choosing trading rates $U=(U_t)_{t \in [0,T]}$ so as to maximize on average the proceeds $W^U_T:=\int_0^T X_t U_tdt$ while taking into account market impact costs $\kappa \int_0^T U_t^2dt$ from her sales. Any inventory remaining $Q^U_T :=q_0-\int_0^T U_tdt$ at time $T$ is marked to market at $Q_T X_T$ and penalized by $\kappa_T Q_T^2$. 

In the framework we study above this corresponds to the choice
\begin{align*}
   Y^U_0=(0,q_0,x_0), 
   \quad dY^U_t = (X_t U_t, -U_t,0) dt + (0,0,1) dX_t, \quad t \in [0,T],
\end{align*}
i.e., $Y^U:=(W^U,Q^U,X)$ with cost functional
\begin{align*}
   L(\Y^U,U)= -W^U_T -Q^U_T X_T + \int_0^T \kappa U^2_t dt + \kappa_T (Q^U_T)^2.
\end{align*}

\begin{table}[h]
\setlength{\tabcolsep}{10pt}
\centering
\begin{tabular}{ c  r | c c c c c c c c c c c }
  &  $H$ &  1/16 &  1/8 &  1/4 &  1/2 &  3/4 &  7/8  \\
\hline 
  \multirow{ 4 }{*}{ \makecell{$\Asig$\\ using \cite{kalsi2020optimal}}}& $N = 1$  & 0.03 & 0.02 & 0.01 & 0.00 & 0.02 & 0.01  \\ 
& $N = 2$  & 1.36 & 0.85 & 0.30 & 0.00 & 0.16 & 0.43  \\ 
& $N = 3$  & 2.46 & 1.45 & 0.44 & 0.00 & 0.20 & 0.41  \\ 
& $N = 4$  & 2.61 & 1.51 & 0.47 & 0.00 & 0.20 & 0.44  \\ 
\hline
\hline
\multirow{ 5 }{*}{ $\Asig$ }& $N = 1$  & 0.03 & 0.02 & 0.01 & 0.00 & 0.00 & 0.01  \\ 
& $N = 2$  & 1.35 & 0.85 & 0.30 & 0.00 & 0.16 & 0.33  \\ 
& $N = 3$  & 2.46 & 1.45 & 0.44 & 0.00 & 0.19 & 0.41  \\ 
& $N = 4$  & 2.55 & 1.51 & 0.47 & 0.00 & 0.20 & 0.40  \\ 
& $N = 5$  & 2.59 & 1.54 & 0.47 & 0.00 & 0.20 & 0.41  \\ 
\hline
\multirow{ 5 }{*}{ $\Alog$ }& $N = 1$  & 0.03 & 0.02 & 0.01 & 0.00 & 0.02 & 0.07  \\ 
& $N = 2$  & 1.36 & 0.87 & 0.30 & 0.00 & 0.20 & 0.43  \\ 
& $N = 3$  & 2.53 & 1.47 & 0.44 & 0.00 & 0.20 & 0.44  \\ 
& $N = 4$  & 2.61 & 1.53 & 0.48 & 0.00 & 0.21 & 0.44  \\ 
& $N = 5$  & 2.69 & 1.54 & 0.49 & 0.00 & 0.21 & 0.43  \\ 
\end{tabular}

\caption{Numerical study of the optimal execution problem with fractional price fluctuations for different Hurst parameters $H$ using strategies in $\Asig$ and $\Alog$ ($I = 2$, \break$q = 30 + \eta_{2, N}$), with various signature truncation levels $N$. 
Presented are \textbf{relative improvements in percent} w.r.t.\ the TWAP strategy $J(U^\circ) \approx 0.9990$.
The fixed model parameters are $q_0=1.0$, $T=1$, $\kappa=0.001$, $\kappa_T=0.1$ and $\sigma=0.02$. An overall time discretization of $\Delta t = 10^{-2}$ was used. The number of training and testing paths was $2^{19}$ and $2^{22}$, respectively. The Monte Carlo resampling error was below $0.001$. 
The first block contains benchmarks obtained with the linearization method from \cite{kalsi2020optimal}.}\label{tab:optimal_execution}
\end{table}

In our numerical experiments, we  followed \cite[Section~5.4]{kalsi2020optimal}, and chose $X=x_0 + \sigma \xi$ with $x_0 = 1.0$,  $\sigma=0.02$ and $\xi$ a fractional Brownian motion.
We tested our methodology for a range of Hurst-parameters $H$ and signature truncation levels $N$. 
The remaining parameters are given by $\kappa=10^{-3}$, $\kappa_T=10^{-1}$, $q_0=1$, and $T=1$.
For comparison, we calculated benchmarks with the method proposed in \cite{kalsi2020optimal}.
As described in Section~\ref{sec:numerical_method}, this method is based on a transformation to a deterministic optimization problem in terms of the expected signature.

Table~\ref{tab:optimal_execution} collects the results of all tested scenarios. For better comparison we do not present the total expected costs, but the relative improvement compared to a trading strategy liquidating at a suitably chosen constant speed $U^\circ \equiv u \in \R$, also called the \emph{time weighted average price} (TWAP) strategy \cite[Section~6.3]{cartea2015algorithmic}. The corresponding optimal rate and associated costs turn out to be
$$u = q_0 \frac{\kappa_T}{\kappa + T\kappa_T}, \qquad J(U^\circ) = x_0 q_0 - q_0^2 \frac{\kappa \kappa_t}{\kappa + T\kappa_T}.$$

Similar to the previous case study, we find an overall improvement when increasing the signature truncation level $N$ and when moving from strategies in $\Asig$ to strategies in $\Alog$. We also verify that the results of $\Asig$ strategies trained with our Monte Carlo method are close to those obtained with the approach proposed in \cite{kalsi2020optimal}. The slightly better performance of the latter method is expected due to the availability of a more efficient optimization procedure when utilizing the quadratic convex structure of the problem. Nevertheless, using deep signature strategies in $\Alog$, we are able to match and, in certain situations, surpass these benchmarks when comparing the same signature truncation levels.

\appendix
\section{Approximation of progressively measurable controls}\label{apx:measurability}

 In this section we relate canonical filtration on $\rps$ to the Borel $\sigma$-algebras of the restricted path space $\rpst$. 
To this end we define the coordinate process $\Z$ on $\rps$ by
$$\Z: [0,T] \times \rps \to G^{\lfloor p \rfloor}(V): \quad (t, \bx)\mapsto \Z_t(\bx):= \bx(t).$$
\begin{lemma}\label{lem:borel_algebra} For all $t\in[0,T]$ it holds
$\sigma(\Z_s \;\vert\; 0 \le s \le t)  = \rho_t^{-1}(\mathcal{B}(\rpst))$, where 
$$\rho_t : \rps \to \rpst  \quad \bx \mapsto \bx\vert_{[0,t]}.$$
\end{lemma}
\begin{proof}
Let $t\in[0,T]$. Clearly $\sigma(\Z_{s} \;\vert\; 0 \le s \le t) = \sigma(\mathcal{C}_t^{T})$, where 
\begin{align*}
\mathcal{C}^{T}_t  :=&~ \Big\{ \big\{\bx\in \rps \;\vert\; (\bx(t_1), \dots, \bx(t_n))\in B \big\}\;\Big\vert\; t_1, \dots, t_n \in [0,t], \quad B \in \mathcal{B}((G^{\lfloor p \rfloor}(V))^n) \Big\}
\end{align*}
are the cylinder subsets of $\rps$ restricted to the time interval $[0,t]$.
One observes that $\mathcal{C}^{T}_t = \rho^{-1}_t(\mathcal{C}_t)$ with $\mathcal{C}_t := \mathcal{C}^{t}_t\subseteq 2^{\rpst}$. 
By the continuity of the map $\rpst \to G^{\integer{p}}(V): \; \bx \mapsto \bx(s)$ for all $s \in [0,t]$ it follows that $\mathcal{C}_t \subset \mathcal{B}(\rpst)$ and thus $$\sigma(\X_{s} \;\vert\; 0 \le s \le t) =  \sigma(\mathcal{C}_t^{T}) =  \rho^{-1}_t(\sigma(\mathcal{C}_t)) \subseteq \rho^{-1}_t(\mathcal{B}(\rpst)).$$

Conversely,  define the system of dissections $\{t^{n}_k \;\vert\; k=0, \dots, n\} := \{t k/n \;\vert\; k=0, \dots, n\}$ of the interval $[0,t]$ and consider the geodesic interpolation map
$$\phi_n: (G^{\integer{p}}(V))^{n} \to \rpst, \quad (\eta_1, \dots, \eta_n) \mapsto  \sum_{k=1}^{n} 1_{[t^n_{k-1},t^n_k)} \Upsilon^{\eta_{k-1}, \eta_{k}}\left(\frac{(\cdot) - t^n_{k-1}}{t^n_k - t^n_{k-1}}\right),$$
where $\Upsilon^{a,b}:[0,1]\to G^{\integer{p}}(V)$ denotes the standard geodesic in $G^{\integer{p}}(V)$ connecting the points $a, b \in G^{\integer{p}}(V)$ (see \cite[Theorem~7.32, Proposition~7.42]{FV10}).
One readily observes that $\phi_n$ is continuous.
Indeed, this continuity is obvious when $G^{\integer{p}}(V)$ is equipped with the geodesic distance, i.e. the Carnot-Caratheodory metric. 
In Section~\ref{sec:notat-basic-defin} we have equipped $G^{\integer{p}}(V)$ with the (inhomogeneous) subspace topology of $T^{\integer{p}}(V)$. This is consistent because the induced topology on $G^{\integer{p}}(V)$ and $\rpst$ coincides with the one induced by the Carnot-Caratheodory metric (see \cite[Section~8.1.3]{FV10}).

For ease of notion we define for any $\bx\in\rpst$ the geodesic interpolation on $\{t^{n}_k\}$ by $\bx^{(n)} := \phi_n(\bx(t_1^{n}), \dots, \bx(t_n^{n}))$.
From Wiener's characterization of geometric rough paths (see \cite[Theorem 8.23]{FV10}) it thus follows that
\begin{align*}
\lim_{n\to\infty} d_{p-\mathrm{var};[0,t]}(\bx^{(n)}, \bx) = 0.
\end{align*}
Hence for any $\bx_0 \in \rpst$ and $r>0$  we have
\begin{align*}
B_r(\bx_0) :=& \Big\{ \bx \in \rpst \;\Big\vert\; d_{p-\mathrm{var};[0,t]}(\bx_0, \bx) \le r \Big\} \\
=& \Big\{ \bx \in \rps \;\Big\vert\; \lim_{n\to\infty} d_{p-\mathrm{var};[0,t]}(\bx_0, \bx^{(n)}) \le r \Big\} \\
=& \bigcap_{k=1}^{\infty}\bigcup_{m=1}^{\infty}\bigcap_{n = m}^{\infty}\Big\{ \bx \in \rpst \;\Big\vert\;  d_{p-\mathrm{var};[0,t]}(\bx_0, \bx^{(n)}) \le r + \frac{1}{k}\Big\} \\
=& \bigcap_{k=1}^{\infty}\bigcup_{m=1}^{\infty}\bigcap_{n = m}^{\infty}\Big\{ \bx \in \rpst \;\Big\vert\;  (\bx(t^{n}_1), \dots, \bx(t^{n}_n)) \in \phi_n^{-1}(B_{r+ \frac{1}{k}}(\bx_0))\Big\}
\end{align*}
Since by continuity of $\phi_n$ it holds $\phi_n^{-1}(B_r(\bx_0)) \in \mathcal{B}((G^{\integer{p}}(V))^n)$, we see that the above right-hand side is the limes inferior of cylinder sets and thus in $\sigma(\mathcal{C}_t)$.
This proves that $\mathcal{B}(\rpst) \subseteq\sigma(\mathcal{C}_t)$ and thus finally $\mathcal{B}(\rpst) = \sigma(\mathcal{C}^{T}_t) = \sigma(\Z_{s} \;\vert\; 0 \le s \le t)$.
\end{proof}

\bibliographystyle{siamplain}

\end{document}